\documentclass[12pt,reqno]{amsart}
\usepackage{amssymb}
\usepackage{amsmath}
\usepackage[square,sort,comma,numbers]{natbib}
\usepackage{enumerate}
\usepackage{mathrsfs}
\usepackage{colonequals}
\usepackage[all]{xy}
\usepackage[table]{xcolor}
\usepackage{comment}
\usepackage{marginnote}
\usepackage[english]{babel}
\usepackage{fullpage}
\usepackage{mathtools}
\usepackage{tikz-cd}
\usepackage{pdflscape}
\usepackage{booktabs}
\usepackage{longtable}

\usepackage{hyperref}

\hypersetup{pdftitle={Definite orders with the cancellation property},pdfauthor={Daniel Smertnig, John Voight}}
\hypersetup{colorlinks=true,linkcolor=blue,anchorcolor=blue,citecolor=blue}

\RequirePackage{mathrsfs} %\let\mathcal\mathcal

\numberwithin{equation}{section} 
\numberwithin{figure}{section}

\newtheorem{theorem}[equation]{Theorem}
\newtheorem{lemma}[equation]{Lemma}
\newtheorem{proposition}[equation]{Proposition}
\newtheorem{corollary}[equation]{Corollary}

\theoremstyle{definition} 
\newtheorem{definition}[equation]{Definition}
\newtheorem{algorithm}[equation]{Algorithm}

\theoremstyle{remark}
\newtheorem{remark}[equation]{Remark}
\newtheorem{example}[equation]{Example}

\newenvironment{enumalph}
{\begin{enumerate}

}
{\end{enumerate}}

\newenvironment{enumroman}
{\begin{enumerate}

}
{\end{enumerate}}

\newcommand{\sbl}{\!{}_{\textup{\textsf{\tiny{L}}}}\,}
\newcommand{\sbr}{\!{}_{\textup{\textsf{\tiny{R}}}}\,}
\newcommand{\calOL}{\calO_{\textup{\textsf{\tiny{L}}}}}
\newcommand{\calOR}{\calO_{\textup{\textsf{\tiny{R}}}}}

\DeclareMathOperator{\M}{M}
\DeclareMathOperator{\Gen}{Gen}
\DeclareMathOperator{\SpGen}{SpGen}
   
\DeclareMathOperator{\disc}{disc}

\DeclareMathOperator{\discrd}{discrd}

\DeclareMathOperator{\id}{id}
\DeclareMathOperator{\rad}{rad}
\DeclareMathOperator{\SpCls}{SpCls}

\DeclareMathOperator{\Cls}{Cls}   % class set
\DeclareMathOperator{\StCl}{StCl}   % class set
\DeclareMathOperator{\Cl}{Cl}       % class group
\DeclareMathOperator{\nrd}{nrd}       % reduced norm
\DeclareMathOperator{\Nm}{Nm}  
\DeclareMathOperator{\mass}{mass}
\DeclareMathOperator{\Idl}{Idl}
\DeclareMathOperator{\PIdl}{PIdl}
\DeclareMathOperator{\coker}{coker}
\DeclareMathOperator{\st}{st}
\DeclareMathOperator{\sgn}{sgn} % sign
\DeclareMathOperator{\rk}{rk} % rank (used for 2-rank of a group)
\newcommand{\val}{v} % valuation

\newcommand{\card}[1]{\#{#1}}
\DeclarePairedDelimiter{\abs}{\lvert}{\rvert}
\newcommand{\stc}[1]{[{#1}]_{\textup{st}}}

\newcommand{\Q}{\QQ}

\newcommand{\calO}{O}

\newcommand\ZZ{\mathbb{Z}}

\newcommand\QQ{\mathbb{Q}}

\newcommand\Z{\mathbb{Z}}
\newcommand{\sL}{\mathsf L}

\newcommand{\fraka}{\mathfrak{a}}
\newcommand{\frakb}{\mathfrak{b}}
\newcommand{\frakc}{\mathfrak{c}}

\newcommand{\frakp}{\mathfrak{p}}

\newcommand{\simeqst}{\simeq_{\textup{st}}}

\newcommand{\fm}{\mathfrak{m}}
\newcommand{\fp}{\mathfrak{p}}
\newcommand{\fq}{\mathfrak{q}}
\newcommand{\fD}{\mathfrak{D}}

\newcommand{\fN}{\mathfrak{N}}

\renewcommand{\leq}{\leqslant}
\renewcommand{\le}{\leqslant}
\renewcommand{\geq}{\geqslant}
\renewcommand{\ge}{\geqslant}

\newcommand{\defi}[1]{\textsf{#1}} 	% for defined terms

\title{Definite orders with locally free cancellation}

\author{Daniel Smertnig}
\address{Department of Pure Mathematics, University of Waterloo, Waterloo, ON, Canada N2L 3G1}
\email{dsmertni@uwaterloo.ca}

\author{John Voight}
\address{Department of Mathematics,
  Dartmouth College, 6188 Kemeny Hall, Hanover, NH 03755, USA}
\email{jvoight@gmail.com}

\begin{document}

\begin{abstract}
We enumerate all orders in definite quaternion algebras over number fields with the Hermite property; this includes all orders with the cancellation property for locally free modules.
\end{abstract}
 
\maketitle

\setcounter{tocdepth}{1}
\tableofcontents

\section{Introduction}

\subsection*{Motivation}

Let $R$ be the ring of integers of a number field $F$ with class group $\Cl R$, 
the group of isomorphism classes of locally principal $R$-modules under tensor product.
In fact, every finitely generated, locally free $R$-module $M$ is of the form $M \simeq \fraka_1 \oplus \cdots \oplus \fraka_n$ where each $\fraka_i$ is a locally principal $R$-module; moreover, we have $\fraka_1 \oplus \cdots \oplus \fraka_n \simeq \frakb_1 \oplus \cdots \oplus \frakb_m$ if and only if $m=n$ and $[\fraka_1 \cdots \fraka_n]=[\frakb_1 \cdots \frakb_m] \in \Cl R$.  In particular, every such $M$ is of the form $M \simeq R^{m} \oplus \fraka$ and the \defi{Steinitz class} $[\fraka] \in \Cl R$ is well-defined on the $R$-module isomorphism class of $M$.  Thus, if $\fraka,\frakb$ are locally principal $R$-modules and $m \geq 0$, then: 
\begin{enumroman}
\item $R^m \oplus \fraka \simeq R^m \oplus \frakb$ if and only if $\fraka \simeq \frakb$; and 
\item $R^m \oplus \fraka$ is a free $R$-module if and only if $\fraka$ is free.  
\end{enumroman}
Property (i) can be seen as a cancellation law.  And in this way, we \emph{rediscover} the group operation on $\Cl R$: given classes $[\fraka],[\frakb] \in \Cl R$, we have $\fraka \oplus \frakb \simeq R \oplus \frakc$ with $[\frakc]=[\fraka\frakb]=[\fraka][\frakb]$.

\subsection*{Generalizations}

We now pursue a noncommutative generalization: we seek to define a group operation on isomorphism classes of modules in a way analogous to the previous section.  Let $\calO$ be an $R$-order in a finite-dimensional semisimple $F$-algebra $B$.  The \defi{(right) class set} $\Cls \calO$ of $\calO$ is the set of locally principal right fractional $\calO$-ideals $I \subseteq B$ under the equivalence relation $I \sim J$ if there exists $\alpha \in B^\times$ such that $I=\alpha J$.  Equivalently, $\Cls \calO$ is the set of isomorphism classes of locally principal right $\calO$-modules.  Because $B$ may be noncommutative, the class set $\Cls \calO$ in general need not be a group---but there is a distinguished element $[\calO] \in \Cls \calO$, and by the geometry of numbers (the Jordan--Zassenhaus theorem; see, for example, Reiner \cite[\S 26]{Reiner75} for a proof), we have $\#\Cls \calO < \infty$, so at least $\Cls \calO$ is a pointed finite set.

In this generality, properties (i) and (ii) above may fail to hold!  Accordingly, we make the following definitions.  Let $M$,~$N$ be finitely generated, locally free right $\calO$-modules.  We say $M$,~$N$ are \defi{stably isomorphic}, written $M \simeqst N$, if there exists $m \in \Z_{\geq 0}$ such that $O^m \oplus M \simeq \calO^m \oplus N$; and we say $M$ is \defi{stably free} if $M \simeqst \calO^m$ for some $m \in \Z_{\geq 0}$.  

Let $\StCl \calO$ denote the set of stable isomorphism classes of locally principal right $\calO$-modules.  Taking the stable isomorphism class gives a natural surjective map of pointed sets, called the \defi{stable class map}
\begin{equation} \label{eqn:clsst}
\begin{aligned}
\st \colon \Cls \calO &\to \StCl \calO \\
[I] &\mapsto \stc{I}.
\end{aligned}
\end{equation}
Moreover, for $\stc{I},\stc{J} \in \StCl \calO$, there exists a unique $\stc{K} \in \StCl \calO$ such that $I \oplus J \simeq \calO \oplus K$ \cite[I, p.~115]{Froehlich75}; and under the operation $\stc{I}+\stc{J}=\stc{K}$, the set $\StCl \calO$ becomes a finite abelian group (related to a Grothendieck group, see Remark \ref{rmk:Grothendieckgroup}).

In the most desirable circumstance, the stable class map is a bijection---with this in mind, we make the following definitions.  We say $\calO$ has \defi{locally free cancellation} if $M \simeqst N$ implies $M \simeq N$ for all finitely generated, locally free right $\calO$-modules $M$,~$N$.
In this case, if $K$, $M$,~$N$ are finitely generated, locally free modules with $K \oplus M \simeq K \oplus N$, then $M \simeq N$.
We say $\calO$ is \defi{Hermite} if $M \simeqst \calO^m$ implies $M \simeq \calO^m$ for all such $M$ and $m \geq 0$.
If $\calO$ has locally free cancellation, then $\calO$ is Hermite; but importantly the converse does \emph{not} hold in general---a counterexample is described in detail by Smertnig \cite{Smertnig15}.

Every finitely generated, locally free right $\calO$-module $M$ is of the form $M \simeq \calO^m \oplus I$ where $I$ is a locally principal right $\calO$-module \cite[I]{Froehlich75}.  Consequently, $\calO$ has locally free cancellation if and only if the stable class map \eqref{eqn:clsst} is injective (and therefore bijective), and $\calO$ is Hermite if and only if $\st^{-1}(\stc{\calO}) = \{[\calO]\}$ (trivial kernel as pointed sets).  If $\calO$ has locally free cancellation, then by transport \eqref{eqn:clsst} defines a natural group structure on $\Cls \calO$.  

The reduced norm defines a surjective map of pointed sets 
\begin{equation} \label{eqn:nrdclsgO}
\begin{aligned} 
\nrd \colon \Cls \calO &\to \Cl_{G(\calO)} R \\
[I] &\mapsto [\nrd(I)]
\end{aligned}
\end{equation}
where $\Cl_{G(\calO)} R$ is a finite abelian group, a certain modified class group of $R$ associated to the idelic normalizer of $\calO$.  By a theorem of Fr\"ohlich \cite[II]{Froehlich75}, extending an earlier result of Swan \cite{Swan62} for maximal orders, there is an isomorphism of groups $\StCl \calO \simeq \Cl_{G(\calO)} R$. In this way, the rather abstract stable class map \eqref{eqn:clsst} can be reinterpreted concretely in terms of the reduced norm. For more, see Curtis--Reiner \cite[\S47]{CurtisReiner87} and Yu \cite{Yu17}.  

Finally, as a consequence of the theorem of strong approximation, if no simple factor of $B$ is a totally definite quaternion algebra, then the map $\nrd$ in \eqref{eqn:nrdclsgO} is bijective, and consequently $\calO$ has locally free cancellation.

\subsection*{Main result}

In light of the previous section, we are left with the case where $B$ is a totally definite quaternion algebra (and in particular $F$ is a totally real field), and we refer to an $R$-order $\calO \subseteq B$ as a \defi{definite (quaternion) order}.  
Our main result is as follows.

\begin{theorem} \label{thm:mainthm}
Up to ring isomorphism, there are exactly $303$ definite Hermite quaternion orders and exactly $247$ with locally free cancellation.
\end{theorem}

The orders in Theorem \ref{thm:mainthm} are listed in appendix \ref{appendix:tables}, along with detailed information about them; a computer-loadable version is available \cite{Smertnig-Voight19:github}.  These orders arise in quaternion algebras over exactly $36$ number fields $F$ up to isomorphism.  If we fix the fields $F$ arising this way, we can refine our count by looking at orders up to $R$-algebra isomorphism, where $R$ is the ring of integers of $F$ (thereby distinguishing Galois conjugates): counted this way, there are exactly $375$ definite Hermite quaternion $R$-orders and exactly $316$ with locally free cancellation.

Theorem \ref{thm:mainthm} can be seen as the culminating resolution to a very general class number $1$ (or unique factorization) problem for central simple algebras over number fields.

\subsection*{Previous results}

Vign\'eras \cite{Vigneras76} showed that there are only finitely many isomorphism classes of definite, hereditary quaternion orders with locally free cancellation.  She provided a numerical criterion characterizing locally free cancellation, but in fact this was shown to characterize Hermite orders by Smertnig \cite{Smertnig15}.  By use of Odlyzko discriminant bounds, Vign\'eras found that Hermite orders are only possible over number fields of degree at most $33$, and then she classified them over quadratic and cyclic cubic fields.  More recently, Hallouin--Maire \cite{HallouinMaire06} classified definite hereditary Hermite orders (they refer to \emph{Eichler} orders, but like Vign\'eras only consider those Eichler orders of \emph{squarefree} reduced discriminant) 
by a rather involved analysis, finding that they arise only for fields of degree at most $6$.  Finally, Smertnig \cite{Smertnig15} completed the classification of definite hereditary orders with locally free cancellation.

The Hermite property has appeared in work by other authors in various guises.  Estes--Nipp
 \cite{EstesNipp89} and Estes \cite{Estes91a} considered factorization properties in quaternion orders, among them one they call \emph{factorization induced by local factorization \textup{(}FLF\textup{)}}.
After observing that the \emph{principal genus} of $\calO$ (as defined by Estes) consists precisely of the stably free right $\calO$-ideals, the theorem of Estes \cite[Theorem 1]{Estes91a} shows that $\calO$ has FLF if and only if $\calO$ is Hermite, a connection first noted by Nipp \cite{Nipp75}.
Estes--Nipp classify all 40 definite Hermite quaternion $\ZZ$-orders \cite[Table I]{EstesNipp89}.
Unlike the results of Vign\'eras and Hallouin--Maire, their classification is not restricted to hereditary orders---but it is only carried out over $R=\Z$.

Every quaternion order $\calO$ with $\# \Cls\calO = 1$ trivially has locally free cancellation, and all orders with $\# \Cls\calO \le 2$ were enumerated by Kirschmer--Lorch \cite{KirschmerLorch16}.  

The global function field case was studied by Denert--Van Geel \cite{DenertVanGeel86,DenertVanGeel88}.
In this case, because of the absence of archimedean places, there are central simple algebras of arbitrary dimension that do not satisfy strong approximation.  Denert--Van Geel show that if $\calO$ is a definite Hermite order over a global ring in a function field $F$, then $F=\mathbb F_q(t)$ must be a rational function field \cite[Theorem 2.1]{DenertVanGeel86}.
Moreover, they prove that---unlike the number field setting---there are infinitely many nonisomorphic quaternion algebras containing definite maximal orders with locally free cancellation \cite[Theorem 2.2]{DenertVanGeel86}.  (Some care must be taken in reading these papers, as they seem to incorrectly identify the Hermite property with locally free cancellation.)

\subsection*{Discussion}

Our proof of Theorem~\ref{thm:mainthm} follows the general approach of Vign\'eras and Hall\-ouin--Maire, but it differs in two important respects.  
\begin{enumerate}
\item Our paper is essentially self-contained, and in particular we do not use the criterion of Vign\'eras (derived from a computation involving Tamagawa measures).  Instead, we give a simple, direct argument (Theorem \ref{thm:mass},  Corollary \ref{cor:equamassHermite}) that the mass of the fibers in the stable class map is \emph{constant}; this allows us to reinterpret the theory in a clarifying way (Proposition \ref{prop:lfcanc}).  (For sanity, we show that our criterion implies the criterion of Vign\'eras, see Theorem \ref{thm:vigneras-criterion}.) Combined with the Eichler mass formula, this provides a stable mass formula for each of the fibers of the stable class map, from which we may proceed with analytic estimates.
\item As much as possible, we use machine computation in lieu of delicate, case-by-case analysis by hand.  This makes it easier for the reader to verify and to experiment with the result \cite{Smertnig-Voight19:github}. It also means that we can get away using slightly weaker, but easier to prove, bounds for the degree of $F$ (Proposition \ref{prop:df-bound}).  This approach is important for reproducibility and to avoid slips, given the complexity of the answer.  Our calculations are performed in the computer algebra system \textsc{Magma} \cite{Magma}; the total running time is less than an hour on a standard CPU.  We hope that our quaternionic algorithms will be of further use to others, beyond the classification in this paper: for example, we implement a systematic enumeration of suborders and superorders of quaternion orders, as well as the computation of the stable class group.
\end{enumerate}

We checked our output by restricting it and comparing to existing lists of orders (see Remark \ref{rmk:wechecked}), and in every case we checked they agree.

The list of definite quaternion orders is quite remarkable!  For some interesting examples, further discussion, and an application to factorization, see section \ref{sec:examappli}.

\subsection*{Organization}

Our paper is organized as follows.  In section \ref{sec:prelim}, we set up preliminaries on class groups and in section \ref{sec:stabcan} the properties of locally free cancellation and Hermite.  In section \ref{sec:characterization}, we characterize these properties in terms of masses.  In section \ref{sec:masses}, we establish a mass formula for the fibers of the stable class map.  Next, in section \ref{sec:bounds}, we present bounds for the search and our algorithm to find all orders with locally free cancellation.  We then discuss examples and applications in section \ref{sec:examappli}.  Finally, in Appendix \ref{appendix:comparvign}, we compare our stable mass formula with the criterion of Vign\'eras, and then in Appendix \ref{appendix:tables} we present the tables describing the output in detail.

\subsection*{Acknowledgements}

The authors would like to thank Emmanuel Hallouin, Christian Maire, Markus Kirsch\-mer, Roger Wiegand, and the anonymous referee.
Voight was supported by an NSF CAREER Award (DMS-1151047) and a Simons Collaboration Grant (550029).
Smertnig was supported by the Austrian Science Fund (FWF) project J4079-N32.
The research for this paper was conducted while Smertnig was visiting Dartmouth College; he would like to extend his thanks for their hospitality.

\section{Preliminaries} \label{sec:prelim}

As a general reference for quaternion algebras, we refer to the books of Vign\'eras \cite{Vigneras80} and Voight \cite{Voight18}.

\subsection*{Notation}

Throughout, let $F$ be a totally real number field of degree $n=[F:\Q]$ and ring of integers $R$.  Let $F_{>0}^\times$ be the set of totally positive elements of $F^\times$ (positive under all real embeddings of $F$).  
For a (nonzero) prime $\frakp$ of $R$, let $F_\fp$ denote the completion of $F$ at $\fp$ with valuation ring $R_\frakp \subseteq F_\frakp$.  We write ${\widehat{F}} \colonequals \prod'_{\fp} F_\fp$ for the \emph{finite} adeles of $F$, the restricted direct product with respect to $R_\frakp$ indexed over the primes of $R$; we similarly write $\widehat R \colonequals \prod_\frakp R_\frakp$.

Let $S$ be a finite (possibly empty) set of primes of $R$.
For a subset $X \subseteq \widehat F$, denote by 
\begin{equation}
X_{S,1} \colonequals \{ \alpha=(\alpha_\frakp)_\frakp \in X : \alpha_\fp=1 \text{ for all $\fp \in S$}\}.
\end{equation}
Let $\Idl R$ be the group of fractional $R$-ideals $\fraka \subseteq F$, and let $\Idl_S R \leq \Idl R$ be the subgroup of those $\fraka \in \Idl R$ for which $\fraka_\fp= R_\fp$ for all $\fp \in S$.  Let
\begin{equation}
F^{\times}_S \colonequals \{ a \in F^\times : \text{$a \in R_\fp^\times$ for all $\fp \in S$}\},
\end{equation} 
and let 
\begin{equation}
\PIdl_S R \colonequals \{ aR : a \in F^\times_S \} \leq \Idl_S R.
\end{equation}

Also throughout, let $B$ be a definite quaternion algebra over $F$, and let $\calO$ be an $R$-order in $B$.  We write $\calO_\fp \colonequals \calO \otimes_R R_\fp$ and $B_\fp \colonequals B \otimes_F F_{\fp}$, and define $\widehat B$ and $\widehat \calO$ similarly as above.  

\subsection*{Class groups}

The following class group will be of central importance for us.  From now on, let $S$ be a finite set of primes $\frakp$ of $R$ containing all those for which $\calO_\fp$ is not maximal.  We recall that if $\calO_\fp$ is maximal, then $\nrd(\calO_\fp^\times) = R_\fp^\times$ \cite[Lemma 13.4.6; {\citealp[Corollaire II.1.7, Th\'eor\`eme II.2.3(1)]{Vigneras80}}]{Voight18}.
% : indeed, either $B_\frakp \simeq \M_2(F_\fp)$ whence $\calO_\frakp \simeq \M_2(R_\fp)$ and 
% see also Reiner \cite[\S12]{Reiner75} for the 
% To see this, recall first $\nrd(B_\fp) = F_\fp$ (\cite[Exercise 14.6]{Reiner75} or \cite[Lemma 13.4.6]{Voight18}).
% Then, by \cite[Theorem 33.1]{Reiner75} it suffices to consider the case where $B_\fp$ is a division ring.
% In this case $B_\fp$ contains a unique maximal order consisting of all integral elements \cite[\S12]{Reiner75} and the claim follows easily; see also \cite[Exercise 23.3]{Voight18}.
Thus, all primes $\fp$ for which $\nrd(\calO_\fp^\times) \subsetneq R_\fp^\times$ are contained in $S$.

\begin{definition} \label{def:classGO}
Define
\begin{equation}
  \begin{aligned}
    F_{S,\calO}^\times &\colonequals \{ a \in F^\times_{>0} : a \in \nrd(\calO_\fp^\times) \text{ for all $\fp \in S$} \},\\
    \PIdl_{S,\calO} R &\colonequals \{ aR : a \in F_{S,\calO}^\times \}
  \end{aligned}
  \end{equation}
  and let
  \begin{equation}
    \Cl_{G(\calO)} R \colonequals \Idl_S R / \PIdl_{S,\calO} R.
  \end{equation}
\end{definition}

Just as there are canonical group isomorphisms
    \begin{equation} \label{l:cliso:cl}
      F^\times \backslash \widehat F^\times\!/ \widehat R^\times  \simeq \Cl R = \Idl R /  \PIdl R  \simeq  \Idl_S R /  \PIdl_S R,
\end{equation}
the `global' definition given in Definition \ref{def:classGO} can also be given equivalently adelically as in the following lemma.  We only use this lemma to identify the stable class group with $\Cl_{G(\calO)} R$ to compare with work of Fr\"ohlich \cite{Froehlich75} later on.

\begin{lemma} \label{l:cliso}
The following statements hold.
  \begin{enumalph}
    \item\label{l:cliso:eq}
    We have
    \[
      F^\times_{>0} \nrd(\widehat \calO^\times) \cap  \widehat F_{S,1}^\times = F^\times_{S,\calO} \widehat R^\times \cap \widehat F_{S,1}^\times.
    \]
  \item\label{l:cliso:clgo} There are canonical group isomorphisms
    \begin{equation*}
      \begin{split}
        \Cl_{G(\calO)} R &\simeq   (F_{>0}^\times \nrd(\widehat \calO^\times) \cap \widehat F_{S,1}^\times) \backslash  \widehat F_{S,1}^\times \\
        &=  (F_{>0}^\times \nrd(\widehat \calO^\times) \cap \widehat F_{S,1}^\times) \backslash \widehat F_{S,1}^\times / \widehat R_{S,1}^\times \simeq F^\times_{> 0} \backslash \widehat F^\times\!/\! \nrd(\widehat \calO^\times).
       \end{split}
     \end{equation*}
  \end{enumalph}
\end{lemma}

\begin{proof}
  First \ref*{l:cliso:eq}.  To show $(\subseteq)$, let $a\beta \in F^\times_{>0} \nrd(\widehat \calO^\times) \cap \widehat F_{S,1}^\times$ with $a \in F^\times_{>0}$ and $\beta \in \nrd(\widehat \calO^\times)$.
  Since $a \beta_\fp = 1$ for all $\fp \in S$, we conclude $a \in \nrd(O_\fp^\times)$.
  Thus $a \in F^\times_{S,\calO}$.  Conversely $(\supseteq)$, let $a \beta \in F^\times_{S,\calO} \widehat R^\times \cap \widehat F_{S,1}^\times$.
  Then again $a \beta_\fp =1$ for all $\fp \in S$.
  Since $a \in \nrd(O_\fp^\times)$, this implies $\beta_\fp \in \nrd(O_\fp^\times)$.
  Since $S$ contains all primes $\fp$ for which $\nrd(O_\fp^\times) \subsetneq R_\fp^\times$, we conclude $\beta \in \nrd(\widehat \calO^\times)$.
  
  Next, \ref*{l:cliso:clgo}.
  The middle equality is due to the obvious inclusion $\widehat R_{S,1}^\times \subseteq F^\times_{>0} \nrd(\widehat \calO^\times) \cap \widehat F_{S,1}^\times$.
  By weak approximation, the homomorphism
  \begin{equation}
    \widehat F_{S,1}^\times \to \widehat F_{S,1}^\times / \widehat R_{S,1}^\times \to F_{>0}^\times \backslash \widehat F^\times\!/ \!\nrd(\widehat \calO^\times)
  \end{equation}
  is surjective with kernel $F_{>0}^\times \nrd(\widehat \calO^\times) \cap \widehat F_{S,1}^\times$, giving the last isomorphism.  
  
  By \ref*{l:cliso:eq} we have $F_{>0}^\times \nrd(\widehat \calO^\times) \cap \widehat F_{S,1}^\times = F^\times_{S,\calO} \widehat R^\times \cap \widehat F_{S,1}^\times$.
  Since $\widehat F_{S,1}^\times / \widehat R_{S,1}^\times \simeq \Idl_S R$, we get an isomorphism
  $(F^\times_{S,\calO} \widehat R^\times \cap \widehat F_{S,1}^\times) \backslash \widehat F_{S,1}^\times / \widehat R_{S,1}^\times \simeq \Idl_S R / \PIdl_{S,\calO} R$.
\end{proof}

\subsection*{Class set and mass}

We now turn to the class set of our order.  A \defi{right fractional $\calO$-ideal} is an $R$-lattice $I \subseteq B$ (a finitely generated $R$-submodule with $IF=B$) such that $\calO \subseteq \calOR(I)$, where $\calOR(I) \colonequals \{\alpha \in B : I\alpha \subseteq I\}$ is the right order of $I$.

\begin{definition}
The \defi{\textup{(}right\textup{)} class set of $\calO$}, denoted $\Cls\calO = \Cls\sbr \calO$, is the set of isomorphism classes of locally principal right fractional $\calO$-ideals.
\end{definition}

Concretely, two right fractional $\calO$-ideals $I$,~$J$ are isomorphic if and only if there exists $\alpha \in B^\times$ such that $\alpha I = J$, and we write $[I] \in \Cls \calO$ for ideal classes.  We define the left class set $\Cls\sbl \calO$ similarly, noting that the standard involution furnishes a bijection between the right and left class sets.  We work primarily on the right, and accordingly suppress the subscript.  

By the geometry of numbers (Jordan--Zassenhaus theorem), we have $\#\Cls \calO<\infty$, so $\Cls \calO$ is a finite pointed set with distinguished element $[O]$---but in general, $\Cls \calO$ is not a group under multiplication.
From the idelic viewpoint, via completions there is a canonical bijection  $\Cls \calO \leftrightarrow B^\times \backslash \widehat B^\times / \widehat \calO^\times$.

The reduced norm induces a surjective map of finite (pointed) sets
\begin{equation} \label{eqn:nrdclso}
\begin{aligned}
  \nrd \colon \Cls\calO &\to \Cl_{G(\calO)} R \\
  [I] &\mapsto [\nrd(I)]
  \end{aligned}
  \end{equation}

\begin{definition} \label{def:stc-fiber}
  For $[\frakb] \in \Cl_{G(\calO)} R$, we define
  \[
    \Cls^{[\frakb]} O \colonequals \nrd^{-1}( \{ [\frakb] \} ) = \big\{ [I] \in \Cls\calO : [\nrd(I)] \in [\frakb] \,\big\}.
  \]
\end{definition}

For a (finite) subset $X=\{[I_i]\}_i \subseteq \Cls (\calO)$, we define the \defi{mass} of $X$ to be
\begin{equation}
  \mass(X) \colonequals \sum_{i=1}^r \frac{1}{[\calOL(I_i)^\times : R^\times]} \in \Q_{\geq 0}
\end{equation}
(well-defined independent of the choice of representatives).  

We will later (Theorem \ref{thm:massformula}) recall an explicit expression for $\mass(\Cls\calO)$, generalizing the Eichler mass formula.

Finally, we define the \defi{genus} of $\calO$ to be the set of $R$-orders in $B$ locally isomorphic to $\calO$.

\section{Locally free cancellation} \label{sec:stabcan}

In this section, we now relate the class set to naturally associated abelian groups. 

\subsection*{Stable isomorphism and cancellation}

Let $M$,~$N$ be right $\calO$-modules.  

\begin{definition}
We say that $M$,~$N$ are \defi{stably isomorphic}, and write $M \simeqst N$, if there exists $m \ge 0$ such that $M \oplus \calO^m \simeq N \oplus \calO^m$.  We say $M$ is \defi{stably free} if $M$ is stably isomorphic to a free module.  
\end{definition}

For more background on stably free modules in a much broader context of (noncommutative) rings, see Lam \cite[Chapter I.4]{Lam06} and McConnell--Robson \cite[Chapter 11]{McConnellRobson01}.  

\begin{remark}
A stably free right $\calO$-module that is not finitely generated is necessarily free \cite[Proposition I.4.2]{Lam06}, so we may restrict our attention to finitely generated modules without loss of generality.
\end{remark}

If $M \simeq N$, then of course $M \simeqst N$; we now give a name to the converse, restricted to locally free modules.

 \begin{definition} \label{def:cancellation}
We say $\calO$ has \defi{locally free cancellation} if for all 
   locally free, finitely generated $\calO$-modules $M$,~$N$, we have
   \begin{equation} M \simeqst N \quad \Rightarrow \quad M \simeq N, \end{equation}
   that is, whenever $M$,~$N$ are stably isomorphic, then they are in fact isomorphic.

We say $\calO$ is a \defi{\textup{(}right\textup{)} Hermite ring} if for all locally free, finitely generated $\calO$-modules $M$ we have $M \simeqst \calO^m$ implies $M \simeq \calO^m$, that is, every stably free, locally free right $\calO$-module is in fact free.  
 \end{definition}
 
A straightforward dualization argument shows that a ring is left Hermite if and only if it is right Hermite.  And  if $\calO$ has locally free cancellation, then clearly $\calO$ is Hermite.

\begin{remark}
Any stably free module over a semilocal ring is free (see Lam \cite[Examples I.4.7]{Lam06} or McConnell--Robson \cite[Theorem 11.3.7]{McConnellRobson01}); hence, stably free right $\calO$-modules are locally free.  Therefore, $\calO$ is Hermite if and only if every stably free right $\calO$-module is free.
\end{remark}

As the following lemma shows, locally free cancellation actually implies the apparently stronger property that arbitrary locally free, finitely generated modules may be cancelled from direct sums, thus justifying the name.

\begin{lemma} \label{l:fullcancellation}
  An order $\calO$ has locally free cancellation if and only if for all locally free, finitely generated $\calO$-module $M$, $N$,~$K$, we have
  \[
    K \oplus M \simeq K \oplus N \quad\Rightarrow\quad M \simeq N.
  \]
  That is, the commutative monoid of isomorphism classes of locally free, finitely generated $\calO$-modules under direct sum is cancellative.
\end{lemma}

\begin{proof}
The implication $(\Leftarrow)$ is clear.  For $(\Rightarrow)$, suppose that $\calO$ has locally free cancellation and $M$, $N$, $K$ are such that $K \oplus M \simeq K \oplus N$.
  Since $K$ is locally free, the fact that $\operatorname{Ext}^{1}$ commutes with localization in our setting shows that $K$ is projective; see, for example, Curtis--Reiner \cite[Proposition 8.19]{CurtisReiner81}.
  Let $K'$ be such that $K \oplus K' \simeq O^n$ for some $n \ge 0$.
  Then $\calO^n \oplus M \simeq \calO^n \oplus N$, and then $M \simeq N$ by locally free cancellation.
\end{proof}

Stable isomorphism defines an equivalence relation on the locally principal right fractional $\calO$-ideals.
We denote by $\stc{I}$ the stable isomorphism class of $I$, and by $\StCl{O}$ the set of all stable isomorphism classes of locally principal (that is, locally free rank $1$) right fractional $\calO$-ideals.  Taking stable isomorphism classes gives a natural surjective map
\begin{equation} \label{eqn:clsst0}
\begin{aligned}
\st \colon \Cls \calO &\to \StCl \calO \\
[I] &\mapsto \stc{I}
\end{aligned}
\end{equation}
of pointed sets, with \defi{kernel} $\Cls^{[R]} \calO$.

For a locally free, finitely generated right $\calO$-module $M$, by definition for all $\fp$ there exists an $m$ such that $M_\fp \cong \calO_\fp^m$.
The number $m$ is independent of $\fp$ and is called the \defi{rank} of $M$.

\begin{lemma} \label{lem:IJstab}
The following statements hold.
\begin{enumalph}
\item Let $M$ be a \textup{(}nonzero\textup{)} locally free, finitely generated right $\calO$-module $M$.  Then there exists a locally principal right $O$-ideal $I$ such that $M \simeq \calO^m \oplus I$ \textup{(}with $m=\rk M-1$\textup{)}, and $\stc{I}$ is uniquely determined by $\stc{M}$.  
\item Let $I$,~$J$ be locally free right fractional $\calO$-ideals.  Then there exists a locally principal right $\calO$-ideal $K$ such that $I \oplus J \simeq O \oplus K$, and $\stc{K}$ is uniquely determined by $\stc{I}$ and $\stc{J}$.
\end{enumalph}
\end{lemma}

\begin{proof}
See Fr\"ohlich \cite[I, p.~115]{Froehlich75} for part (a); part (b) follows from (a), and it is precisely the notion of stable isomorphism that implies that the class is well-defined.
\end{proof}

By Lemma \ref{lem:IJstab}(b), defining $\stc{I} + \stc{J} = \stc{K}$, the set $\StCl{O}$ has the structure of a finite abelian group, the \defi{stable class group} of $\calO$.  
See Reiner \cite[\S35]{Reiner75} or Curtis--Reiner \cite[\S49]{CurtisReiner87} for further detail.

\begin{remark} \label{rmk:Grothendieckgroup}
Let $K_0(\calO)$ be the Grothendieck group of locally free right $\calO$-modules of finite rank. Then there is an exact sequence of groups
\begin{equation}
\begin{aligned}
0 \to \StCl(\calO) &\to K_0(\calO) \xrightarrow{\rk} \ZZ \to 0 \\ 
\stc{I} &\mapsto (I)-(\calO) 
\end{aligned}
\end{equation}

Let $\mathcal V(\calO)$ be the commutative monoid consisting of isomorphism classes $[M]$ of locally free, finitely generated $\calO$-modules together with the operation $[M]+[N] = [M \oplus N]$.
By Lemma~\ref{l:fullcancellation}, if $\calO$ has locally free cancellation, then $\mathcal V(\calO)$ is \emph{cancellative}, and thus embeds into $K_0(\calO)$, with $K_0(\calO)$ in fact being the quotient group of $\mathcal V(\calO)$.
\end{remark}

\subsection*{Reduced norms and equivalences}

Define
\[
  \widehat B^1 \colonequals \{ \widehat \alpha \in \widehat B : \nrd(a)=1 \}.
\] 
We recall the following local description of the stable class group.

\begin{theorem}[{Fr\"ohlich \cite[II, p. 115]{Froehlich75}}] \label{thm:frohlichstcl}
There exists a group isomorphism
\begin{equation}
\begin{aligned}
  \StCl(O) &\simeq B^\times \backslash \widehat B^\times\!/ \widehat B^1 \widehat \calO^\times \\
\stc{I} &\mapsto B^\times \widehat{\alpha} \widehat B^1 \widehat O^\times, 
\end{aligned}
\end{equation} 
where $\widehat{\alpha}=(\alpha_\frakp)_{\frakp}$ if $I_\frakp = \alpha_\frakp O_\frakp$.  
\end{theorem}

In particular, Theorem \ref{thm:frohlichstcl} yields a commutative diagram
\begin{equation}  \label{eqn:clsosimstcl}
  \begin{tikzcd}
    \Cls O \ar[r]{st} \ar[d,"\rotatebox{90}{\(\sim\)}"] & \StCl(O) \ar[d,"\rotatebox{90}{\(\sim\)}"] \\
    B^\times \backslash \widehat B^\times \!/ \widehat O^\times \ar[r] & B^\times \backslash \widehat B^\times\!/\widehat B^1 \widehat \calO^\times
  \end{tikzcd}
\end{equation}
where the top map is the stable class map \eqref{eqn:clsst0} and the bottom map is the natural projection.  

\begin{corollary} \label{c:stcl}
\begin{equation}
  \StCl(O) \simeq F^\times \backslash \widehat F^\times\!/\!\nrd(\widehat \calO^\times)  \simeq \Cl_{G(\calO)} R.
\end{equation}
\end{corollary}

\begin{proof}
Apply the reduced norm to the right-hand side of \eqref{eqn:clsosimstcl} (an isomorphism of groups onto its image) and then apply Lemma~\ref{l:cliso}.
\end{proof}

Since $\Cl_{G(\calO)} R$ is the most accessible of these groups, in the remainder of the paper we will always work with it.

As a consequence of these isomorphisms, we can interpret the fibers of the reduced norm map $\nrd \colon \Cls\calO \to \Cl_{G(\calO)} R$ as parametrizing the isomorphism classes of locally principal right fractional $\calO$-ideals in a given stable isomorphism class.

 \section{Characterization}  \label{sec:characterization}
 
In this section, we characterize the locally free cancellation and Hermite properties in terms of masses and compare the two properties.
 
\subsection*{Masses}
 
We begin with a quick lemma.  
  
\begin{lemma} \label{lem:mass-equiv}
  The following are equivalent:
  \begin{enumroman}
  \item \label{prop:mass-equiv:expr}
    For all $[\frakb] \in \Cl_{G(\calO)} R$,
    \[
      \mass(\Cls^{[\frakb]}  O) = \frac{\mass(\Cls\calO)}{\card{\Cl_{G(\calO)} R}};
    \]
  \item \label{prop:mass-equiv:indep} $\mass(\Cls^{[\frakb]} O)$ is independent of the class $[\frakb] \in \Cl_{G(\calO)} R$; and
  \item \label{prop:mass-equiv:genus} $\mass(\Cls^{[R]} O) = \mass(\Cls^{[R]} \calO')$ for every order $\calO'$ that is locally isomorphic to $\calO$.
  \end{enumroman}
\end{lemma}

\begin{proof}
The equivalence
  \ref*{prop:mass-equiv:expr}${}\Leftrightarrow{}$\ref*{prop:mass-equiv:indep} follows from the surjectivity of \eqref{eqn:nrdclso}.

  To set up the equivalence \ref*{prop:mass-equiv:indep}${}\Leftrightarrow{}$\ref*{prop:mass-equiv:genus}, let $\calO'$ be locally isomorphic to $\calO$.
  Then there exists a locally principal right $O$-ideal that is also a locally principal left $\calO'$-ideal \cite[Lemma 17.4.6; {\citealp[Lemme I.4.10]{Vigneras80}}]{Voight18}.
  Let $[\frakc] = [\nrd(J)] \in \Cl_{G(\calO)} R$.
  The ideal $J$ induces a bijection
  \begin{equation}
  \begin{aligned}
    \mu\colon \Cls\calO &\xrightarrow{\sim} \Cls\calO' \\
    [I] &\mapsto [IJ^{-1}].
  \end{aligned}
  \end{equation}
  Noting that also $\calOL(IJ^{-1}) = \calOL(I)$, we see that $\mu$ preserves the mass of each class.
  Since $\nrd(IJ^{-1}) = \nrd(I)\nrd(J)^{-1}$ and $\Cl_{G(\calO)} R = \Cl_{G(O')} R$, this in turn induces mass-preserving bijections $\Cls^{[\frakb]} O \to \Cls^{[\frakb \frakc^{-1}]} \calO'$.

  For the direction \ref*{prop:mass-equiv:indep}${}\Rightarrow{}$\ref*{prop:mass-equiv:genus}, we take $J$ to be a connecting $\calO',\calO$-ideal and conclude from the previous paragraph that
  \[
    \mass(\Cls^{[R]} \calO') = \mass(\Cls^{[\frakc]} O) = \mass(\Cls^{[R]} O).
  \]
  For the converse \ref*{prop:mass-equiv:indep}${}\Leftarrow{}$\ref*{prop:mass-equiv:genus}, let $J$ be a locally principal right $\calO$-ideal  with $[\nrd(J)]=[\frakb]$, and let $\calO'=\calOL(J)$; then
  \[
    \mass(\Cls^{[\frakb]} O) = \mass(\Cls^{[R]} \calO') = \mass(\Cls^{[R]} O). \qedhere
  \]
\end{proof}

We now prove the promised characterizations.

 \begin{proposition} \label{prop:lfcanc}
   The following are equivalent.
   \begin{enumroman}
   \item\label{lfc:module} $\calO$ has locally free cancellation.
   \item\label{lfc:full} The monoid of isomorphism classes of locally free, finitely generated right $\calO$-modules is cancellative.
   \item\label{lfc:ideal} Locally principal right \textup{(}fractional\textup{)} $\calO$-ideals that are stably isomorphic are isomorphic.
   \item\label{lfc:fiber} The map $\nrd\colon \Cls\calO \to \Cl_{G(\calO)} R$ is injective, and hence bijective.
   \item\label{lfc:fibermass1}
     For every $[\frakb] \in \Cl_{G(\calO)} R$ and $I$ a right \textup{(}fractional\textup{)} $O$-ideal with $[\nrd(I)]=[\frakb]$, 
     \[
       \mass(\Cls^{[\frakb]} O) = \frac{1}{[\calOL(I)^\times:R^\times]}.
     \]
   \item\label{lfc:card} We have $\card{\Cls\calO} = \card{\Cl_{G(\calO)} R}$.
   \end{enumroman}
   Moreover, these statements are equivalent to the corresponding statements for left $\calO$-modules.
 \end{proposition}

 \begin{proof}
   \ref*{lfc:module}${}\Leftrightarrow{}$\ref*{lfc:full} by Lemma~\ref{l:fullcancellation}.
   \ref*{lfc:module}${}\Rightarrow{}$\ref*{lfc:ideal} holds by definition, and   \ref*{lfc:ideal}${}\Rightarrow{}$\ref*{lfc:module} follows from Lemma \ref{lem:IJstab}: if $M,N$ are nonzero stably isomorphic, locally free right $\calO$-modules of finite rank, writing $M \simeq \calO^m \oplus I$ and $N \simeq \calO^m \oplus J$, we have $I \simeqst J$, so by hypothesis $I \simeq J$, hence $M \simeq N$.
   
   For the equivalence \ref*{lfc:ideal}${}\Leftrightarrow{}$\ref*{lfc:fiber}, 
   two locally principal right fractional $\calO$-ideals are isomorphic if and only if $[I]=[J]$ in $\Cls\calO$, and they are stably isomorphic if and only if $[\nrd(I)]=[\nrd(J)]$ in $\Cl_{G(\calO)} R$ in view of Corollary~\ref{c:stcl}.

   The equivalences \ref*{lfc:fiber}${}\Leftrightarrow{}$\ref*{lfc:fibermass1} is direct, since $\nrd([I])=[\frakb]$. The equivalence \ref*{lfc:fiber}${}\Leftrightarrow{}$\ref*{lfc:card} follows as $\nrd$ is surjective.

   Finally, the corresponding equivalences hold for left modules, and the left-right symmetry follows from $\card{\Cls\sbr O} = \card{\Cls\sbl O}$ and \ref*{lfc:card}.
 \end{proof}

 \begin{proposition} \label{prop:sff}
   The following are equivalent.
   \begin{enumroman}
   \item\label{sff:module} $\calO$ is a Hermite ring.
   \item\label{sff:ideal} Every stably free right $\calO$-ideal is free.
   \item\label{sff:fiber} The kernel of the stable class map is trivial, that is, $\#\Cls^{[R]}(O)=1$.
   \item\label{sff:mass} We have
     \[
       \mass(\Cls^{[R]} O) = \frac{1}{[O^\times:R^\times]}.
     \]
   \end{enumroman}
   Moreover, these statements are equivalent to the corresponding statements for left $\calO$-modules.
 \end{proposition}

 \begin{proof}
 As in Proposition \ref{prop:lfcanc}.
 \end{proof}

\subsection*{Further remarks}

We give a few remarks on extensions and other characterizations of the above.  

 \begin{remark}
   With the exception of Propositions \ref{prop:lfcanc}\ref{lfc:fibermass1} and \ref{prop:sff}\ref{sff:mass} (replace with a volume) and with suitable changes to the definition of $\Cl_{G(\calO)} R$ to account for the archimedean places, similar characterizations hold for indefinite orders over number fields.  But for indefinite orders, strong approximation already implies that the map $\nrd\colon \Cls\calO \to \Cl_{G(\calO)} R$ is bijective, so we always have locally free cancellation.
 \end{remark}

 \begin{remark}
   By yet another characterization, the ring $\calO$ is a Hermite ring if and only if every unimodular row (or column) of rank $n \geq 1$ can be extended to an invertible $n\times n$-matrix; that is, $\calO$ is Hermite if and only if the \emph{general linear rank} of $\calO$ is $1$ \cite[11.1.14]{McConnellRobson01}.
   \end{remark}
   
   \begin{remark}
   Since $\calO$ has Krull dimension $1$, its stable rank (and hence its general linear rank), is at most $2$ \cite[11.3.7]{McConnellRobson01}.
   From this it follows---whether or not $\calO$ is Hermite---that every stably free module of rank $\ge 2$ is free, a fact which can also be seen by an application of strong approximation.  
   
   Moreover, if $M,N$ are locally free right $\calO$-modules of rank $m\ge 2$ and $M \simeqst N$, then  $M \simeq N$.
   Indeed, writing $M=\calO^{m-1} \oplus I$ with $I$ a locally principal right $\calO$-ideal, the claim follows from McConnell--Robson \cite[11.4.8]{McConnellRobson01}.
   This gives another justification for \ref*{sff:ideal}$\Rightarrow$\ref*{sff:module} in Propositions \ref{prop:lfcanc} and \ref{prop:sff}.
   For $M$, $N$ of rank $1$, we still get $M \oplus O \simeq N \oplus O$, but of course we might not be able to cancel the final factor $\calO$.
 \end{remark}
 
  There is a stronger connection between the two properties than it might initially seem, as the following result shows.

 \begin{proposition} \label{prop:canc-sff}
   The following statements are equivalent.
   \begin{enumroman}
   \item \label{canc-sff:canc} $\calO$ has locally free cancellation.
   \item \label{canc-sff:sff} Every order $\calO'$ locally isomorphic to $\calO$ is Hermite.
   \end{enumroman}
 \end{proposition}
 
 \begin{proof}
   First we prove \ref*{canc-sff:canc}${}\Rightarrow{}$\ref*{canc-sff:sff}.
   Let $\calO'$ be locally isomorphic to $\calO$.
   As in the proof of Lemma \ref{lem:mass-equiv}, there exists a locally principal right $\calO$-ideal that is also a locally principal left $\calO'$-ideal.
   Let $J$ be a stably free right $\calO'$-ideal.
   Then the compatible product $JI$ is a locally principal right $\calO$-ideal with $\nrd(J I) = \nrd(J)\nrd(I)$ \cite[Chapter 16]{Voight18}.
   Thus $[\nrd(J I)] = [\nrd(J)] [\nrd(I)] \in \Cl_{G(\calO)} R$.
   Since $\calO$ and $\calO'$ are locally isomorphic, we have $\PIdl_{S,\calO} R = \PIdl_{S,O'} R$ and hence can identify $\Cl_{G(\calO)} R = \Cl_{G(O')} R$.
   Because $J$ is stably free as right $\calO'$-ideal, we have $[\nrd(J)] = [R]$, and conclude $[\nrd(JI)] = [\nrd(I)]$.
   Since $\calO$ has locally free cancellation, then $J I \simeq I$ as right $\calO$-ideals, which implies that $J$ is a principal right $\calO'$-ideal.

   Next we prove \ref*{canc-sff:sff}${}\Rightarrow{}$\ref*{canc-sff:canc}.
   Let $I,J$ be stably isomorphic, locally principal right $\calO$-ideals.
   Let $\calO' \colonequals \calOL(J)$.
   Then $\calO'$ is locally isomorphic to $\calO$, and hence a Hermite ring by hypothesis.
   The inverse $J^{-1}$ is a fractional left $\calO$-ideal with $\nrd(J^{-1}) = \nrd(J)^{-1}$.
   Hence, $IJ^{-1}$ is a compatible product with $[\nrd(IJ^{-1})] = [\nrd(I)][\nrd(J)]^{-1} = [R] \in \Cl_{G(O')} R$.
   Thus, $IJ^{-1}$ is stably free, and by assumption on $\calO'$, $IJ^{-1}=\alpha \calO'$ is free with $\alpha \in B^\times$.
   Multiplying on the right by $J$, we see $I=\alpha J$ so $I \simeq J$ as right $\calO$-ideals.
 \end{proof}

 \section{Masses and suborders} \label{sec:masses}

 Our goal in this section is to investigate the behavior of $\mass(\Cls^{[R]} O)$ and $\Cl_{G(\calO)} R$ as the order $\calO$ varies.  The main result is Theorem~\ref{thm:mass}, which in particular implies that if $\calO$ and $\calO'$ are locally isomorphic orders, then 
 \begin{equation} 
 \mass(\Cls^{[R]} \calO) = \mass(\Cls^{[R]} \calO') 
 \end{equation}
 from which we deduce
 \begin{equation} 
 \mass(\Cls^{[R]} \calO) = \frac{\mass(\Cls \calO)}{\card{\Cl_{G(\calO)} R}}.
 \end{equation} 
 This opens the way for our computations starting in the next section, since $\mass(\Cls\calO)$ can be computed using the Eichler mass formula.
 
\subsection*{Extension of ideals}

We now compare masses a bit more carefully in preparation for our main result.  

Let $O \subseteq \calO'$ be orders, and let $\rho=\rho_{O,O'}$ be the extension map from locally principal right $\calO$-ideals to locally principal right $\calO'$-ideals defined by $\rho(I)=IO'$.  The fibers of $\rho$ have finite cardinality.  Let 
\begin{equation}
\begin{aligned}
\overline \rho \colon \Cls\calO &\to \Cls\calO' \\
 [I] &\mapsto [IO']
 \end{aligned}
 \end{equation} 
 denote the induced map on right ideal classes.  The extension map is compatible with the reduced norm map: let $\varphi=\varphi_{G(\calO), G(O')}\colon \Cl_{G(\calO)} R \to \Cl_{G(O')} R$ denote the canonical epimorphism, then the diagram
\begin{equation} \label{eqn:OOwithCl}
  \begin{tikzcd}
    \Cls\calO' \arrow[r, "\nrd"] & \Cl_{G(O')} R \\
    \Cls\calO \arrow[r, "\nrd"] \arrow[u, "\rho_{O,O'}"] & \Cl_{G(\calO)} R  \arrow[u, "\varphi_{G(\calO),G(O')}", swap].
  \end{tikzcd}
\end{equation}
is commutative.  Moreover, the right-hand map (a natural projection) only depends on the genera of $\calO$ and $\calO'$.  
 
 The map $\overline{\rho}$ is fiber-by-fiber surjective, as follows.  
 (See also Fr\"ohlich \cite[VIII]{Froehlich75}.)
 
\begin{lemma} \label{lem:fiber-surjective}
  Let $[\frakb] \in \Cl_{G(\calO)} R$.
  Then $\overline \rho|_{\Cls^{[\frakb]} O} \colon \Cls^{[\frakb]} O \to \Cls^{\varphi([\frakb])} \calO'$ is surjective, and $\overline{\rho}$ is surjective.
\end{lemma}
  
\begin{proof}
Let $[I'] \in \Cls^{\varphi([\frakb])}(O')$ so $I'$ is a locally principal fractional right $\calO'$-ideal such that $\nrd(I')=\frakb$.  Write $I'_\frakp=\alpha_\frakp \calO_\frakp'$ for each prime $\frakp$.  Since $\calO_\frakp=\calO_\frakp'$ for all but finitely many $\frakp$, by the local--global dictionary for lattices, there exists an $R$-lattice $I$ such that $I_\frakp=\alpha_\frakp \calO_\frakp$ for all $\frakp$, so $[I] \in \Cls^{[\frakb]}(\calO)$ has $[I\calO']=[I']$ as desired.
Running over $[\frakb]$, we conclude $\overline \rho$ is surjective.
\end{proof}

We deduce the following important result for our classification.
 
 \begin{corollary} \label{cor:canc-overorder}
   Let $O \subseteq \calO'$ be orders.
   If $\calO$ is Hermite or has locally free cancellation, then the same is true for $\calO'$.
 \end{corollary}

 \begin{proof}
 By Proposition \ref{prop:sff}(iii), the order $\calO$ is Hermite if and only if $\card{\Cls^{[R]} O} = 1$.  By Lemma~\ref{lem:fiber-surjective}, the map $\Cls^{[R]} O \to \Cls^{[R]} \calO'$ is surjective.  Hence, if $\calO$ is Hermite, so is $\calO'$.

A similar proof works for locally free cancellation.
 \end{proof}

\subsection*{Comparison of masses}

We now look more carefully at the extension map to compare masses.  Let $I'$ be a locally principal right $\calO'$-ideal.
Suppose that there exists a prime $\fp$ such that $\calO_\fq=O'_\fq$ for all primes $\fq \neq \fp$.  Then there exists a transitive right action of $\calO'^\times_\fp$ on $\rho^{-1}(I')$ as follows.  Let $\mu \in \calO'^\times_\fp$.
If $I \in \rho^{-1}(I')$ and $I_\fp=\beta_\fp O_\fp$ with $\beta_\fp \in O_\fp$, we assign to $I$ the unique locally principal right $\calO$-ideal $I \langle \mu \rangle$ such that  $I \langle \mu \rangle = \beta_\fp \mu O_\fp$ and $I\langle \mu \rangle_\fq = I_\fq$ for $\fq \ne \fp$.  The stabilizer of this action is $O_\fp^\times$, hence $\card{\rho^{-1}(I')} = [O'^\times_\fp \colon \calO^\times_\fp]$.  

Repeating this argument over the classes $[I']$, we obtain a bijection
\begin{equation}
  \Cls\calO \leftrightarrow \bigsqcup_{[I'] \in \Cls\calO'} \calOL(I')^\times \backslash\, \rho^{-1}(I').
\end{equation}
See Voight \cite[\S 26.6]{Voight18} for more discussion.

\begin{lemma} \label{lem:doesnotdepend}
If $\calO \subseteq \calO'$ are $R$-orders, then the index $[\widehat\calO'^\times : \widehat\calO^\times]$ is well-defined as a function of the index $[\calO:\calO']$ and the genus of $\calO$ and $\calO'$.  
\end{lemma}

\begin{proof}
We may compute the index locally, so let $\fp$ be prime and let $\fp^m \calO'_\fp \subseteq O_\fp$.  Then
  \begin{equation} \label{eqn:oonotdep}
    [O_{\fp}'^\times: O_{\fp}^\times] = \frac{[O_{\fp}'^\times: 1 + \fp O_{\fp}'][1+\fp O_{\fp}':1+\fp^m O_\fp']}{[O_{\fp}^\times : 1 + \fp O_{\fp}][1+\fp O_{\fp}: 1 + \fp^m O_\fp']} = \frac{[O_{\fp}'^\times: 1 + \fp O_{\fp}']}{[O_{\fp}^\times : 1 + \fp O_{\fp}]} [O_{\fp}' : O_{\fp}].
   \end{equation}
The result follows.
   \end{proof}

\begin{lemma} \label{lemma:mass-fiber}
  Let $O \subseteq \calO'$ be orders.
  \begin{enumalph}
  \item  Let $I'$ be a locally principal right $\calO'$-ideal.
  Then
  \[
    \mass\bigl(\overline{\rho}^{-1}([I'])\bigr)  = [\widehat \calO'^\times\colon \widehat \calO^\times] \mass([I']) .
  \]
  \item We have
  \[ \mass(\Cls\calO) = [\widehat \calO'^\times : \widehat \calO^\times] \mass(\Cls\calO'). \]
  \item We have
  \[
    \mass(\Cls^{[R]} \calO')   = [\widehat \calO'^\times:\widehat \calO^\times]^{-1}\sum_{[\frakb] \in \varphi^{-1}([R])} \mass(\Cls^{[\frakb]} O).
  \]
  \end{enumalph}
\end{lemma}

\begin{proof}
  It suffices to show the claim for the case where $\calO$ and $\calO'$ only differ at a single prime ideal $\fp$.
  For (a), we have 
  \[
    \begin{split}
      \mass\bigl(\overline \rho^{-1}([I'])\bigr)
      &= \sum_{[I] \in \overline{\rho}^{-1}([I'])}  \frac{1}{[\calOL(I)^\times \colon R^\times]} = \sum_{I \in \rho^{-1}(I')} \frac{1}{[\calOL(I)^\times:R^\times] [\calOL(I')^\times \colon \calOL(I)^\times]} \\
      &= \sum_{I \in \rho^{-1}(I')} \frac{1}{[\calOL(I')^\times\colon R^\times]} = \card{\rho^{-1}(I')} \mass([I']) = [O'^\times_\fp \colon \calO^\times_\fp ] \mass([I']).
    \end{split}
  \]
  Summing over the classes $[I']$ and surjectivity then gives (b).  
  
  For part (c), we apply (a) to get
  \[
    [\widehat \calO'^\times:\widehat \calO^\times]\mass(\Cls^{[R]} \calO')   = \mass\bigl(\overline \rho^{-1}( \Cls^{[R]} \calO')\bigr).
  \]
  But
  \[
    \overline \rho^{-1}( \Cls^{[R]} \calO' )
    = \{ [I] \in \Cls\calO \mid [\nrd(IO')]=[R] \in \Cl_{G(O')} R \}
    = \bigsqcup_{[\frakb] \in \varphi^{-1}([R])} \Cls^{[\frakb]} O. \qedhere
  \]
\end{proof}

With the basic comparison of masses in hand, we prove the following key theorem.

\begin{theorem}[Stably free mass] 
  \label{thm:mass}
Let $\calO$ be an $R$-order in $B$.  Then the following statements hold.  
  \begin{enumalph}
  \item\label{mass:indep} If $\calO'$ is locally isomorphic to $\calO$, then $\mass(\Cls^{[R]} O) = \mass(\Cls^{[R]} \calO')$.
  \item\label{mass:stablyfreemass} We have 
    \begin{equation} \label{eq:goal}
      \mass(\Cls^{[R]} O) = \frac{\mass(\Cls\calO)}{\card{\Cl_{G(\calO)} R}}.
    \end{equation}
  \item\label{mass:suborder}
    If $O' \supseteq \calO$ is a superorder, then
    \[
      \mass(\Cls^{[R]} O) = \frac{\card{\Cl_{G(O')} R}}{\card{\Cl_{G(\calO)} R}} [\widehat \calO'^\times: \widehat \calO^\times] \mass(\Cls^{[R]} \calO').
    \]
  \end{enumalph}
\end{theorem}

\begin{proof}
  We begin with \ref*{mass:indep}.  Let $\calO_0 \colonequals \calO \cap \calO'$.  With respect to the inclusion $\calO_0 \subseteq \calO$, Lemma \ref{lemma:mass-fiber}(c) expresses $\mass(\Cls^{[R]} O)$ in terms of masses on $\calO_0$ and the index $[\widehat\calO^\times : \widehat\calO_0^\times]$.  Since $\calO$ is locally isomorphic to $\calO'$, they have equal reduced discriminants $\discrd(\calO)=\discrd(\calO')$ so 
  \[ \discrd(\calO_0)=[\calO:\calO_0]_R \discrd(\calO)=[\calO':\calO_0]_R \discrd(\calO') \]
  so $[\calO:\calO_0]=[\calO':\calO_0]$.  Thus, by Lemma \ref{lem:doesnotdepend} we get the same expression for $\mass(\Cls^{[R]} O')$.
   
Part  \ref*{mass:stablyfreemass} then follows from Lemma~\ref{lem:mass-equiv} together with part \ref*{mass:indep}, and then part \ref*{mass:suborder} follows from Lemma~\ref{lemma:mass-fiber}(b) and part \ref*{mass:stablyfreemass}.
 \end{proof}
 
 \begin{corollary} \label{cor:equamassHermite}
 $\calO$ is Hermite if and only if 
\begin{equation}
\mass(\Cls\calO) = \frac{\card{\Cl_{G(\calO)} R}}{[O^\times:R^\times]}.
\end{equation}
\end{corollary}

\begin{proof}
Combine Proposition \ref{prop:sff}(iv) which says
\[
  \mass(\Cls^{[R]} O) = \frac{1}{[O^\times:R^\times]}; \]
with Theorem \ref{thm:mass}(b).
\end{proof}

 \subsection*{Class groups}

Retracing a standard argument for ray class groups, we determine the kernel of the epimorphism $\Cl_{G(\calO)} R \to \Cl R$ and compare $\Cl_{G(\calO)} R$ and $\Cl_{G(O')} R$ for orders $O \subseteq \calO'$.  We use these in the bounds and algorithms in the next section.  

\begin{lemma} \label{lem:hom-seq}
  Let
  \[
    \begin{tikzcd}[column sep=small]
      A \ar[r, "f"] & B \ar[r,"g"] & C
    \end{tikzcd}
  \]
  be a pair of homomorphisms of abelian groups.
  Then there is an exact sequence
  \[
    \begin{tikzcd}[column sep=small]
      \ker(f) \ar[r] & \ker(g \circ f) \ar[r] & \ker(g) \ar[r] & \coker(f) \ar[r] & \coker(g \circ f) \ar[r] & \coker(g) \ar[r] & 0 
    \end{tikzcd}
  \]
\end{lemma}

\begin{proof}
  This is a consequence of the snake lemma applied to
  \[
    \begin{gathered}[b]
      \begin{tikzcd}
        & A \ar[r,"f"] \ar[d, "g \circ f"] & B \ar[d, "g"] \ar[r] &  B/f(A) \ar[d] \ar[r] & 0 \\
        0 \ar[r] & C \ar[r, "\id"] & C \ar[r] & 0.
      \end{tikzcd}\\[-\dp\strutbox]
    \end{gathered}\qedhere
  \]
 \end{proof}

By Hensel's Lemma, $1 + \fp^l R_\fp \subseteq R_\fp^{\times 2}$ for $l > \val_\fp(4)$.
Therefore, $1 + \fp^l R_\fp \subseteq \nrd(O_\fp^\times)$.
We make use of this in applying weak approximation in the following two proofs.
 
 \begin{proposition} \label{prop:cgrp-sequence}
   There is an exact sequence
   \[
     \begin{tikzcd}[column sep=small]
       1 \ar[r] & R^\times \cap F^\times_{S,\calO} \ar[r] & R^\times \ar[r] & F^\times_S / F^\times_{S,\calO} \ar[r] & \Cl_{G(\calO)} R \ar[r] & \Cl R \ar[r] & 1.
     \end{tikzcd}
   \]
and
   \[
     F^\times_S / F^\times_{S,\calO} \simeq \{\pm 1\}^n \times \widehat R^\times\!/\!\nrd(\widehat \calO^\times).
   \]
 \end{proposition}

 \begin{proof}
   Consider the homomorphisms
   \[
     \begin{tikzcd}
       F^\times_{S,\calO} \ar[r, "f"] & F^\times_S \ar[r, "g"] & \Idl _S R.
     \end{tikzcd}
   \]
   Then $\ker(f) = \mathbf 1$, $\ker(g) = R^\times$, and $\ker(g \circ f ) = R^\times \cap F^\times_{S,\calO}$.
   For the cokernels we have $\coker(g) = \Cl R$, $\coker(g \circ f) = \Cl_{G(\calO)} R$, and $\coker(f) = F^\times_S / F^\times_{S,\calO}$.
   Thus, the previous lemma yields the claimed exact sequence.

   It remains to show $F^\times_S / F^\times_{S,\calO} \simeq \{\pm 1\}^n \times \prod_{\fp \in S} R_\fp^\times\!/\!\nrd(O_\fp^\times)$.
   Let $a \in F^\times_S$.
   Then there exist $b$,~$c \in R \setminus \{0\}$ such that $a=b/c$ and $\val_\fp(b)=\val_\fp(c)=0$ for all $\fp \in S$.
   Observing that the following  is independent of the representatives $b$, $c$, we define
   \begin{equation}
   \begin{aligned}
     f\colon F^\times_S &\to \{\pm 1\}^n \times \prod_{\fp \in S} R_\fp^\times \!/\!\nrd(O_\fp^\times)\\
      a &\mapsto (\sgn a, bc^{-1})
      \end{aligned}
      \end{equation}
   By weak approximation, $f$ is surjective.
   Clearly $\ker f = F^\times_{S,\calO}$.
 \end{proof}

 \begin{proposition} \label{prop:cgrp-sequence-suborder}
   Let $O \subseteq \calO'$ be orders.
   Then there is an exact sequence
   \[
     \begin{tikzcd}[column sep=small]
       \mathbf 1 \ar[r] & R^\times \cap F^\times_{S,\calO} \ar[r] & R^\times  \cap F^\times_{S,O'} \ar[r] & F^\times_{S,O'} / F^\times_{S,\calO} \ar[r] & \Cl_{G(\calO)} R \ar[r] & \Cl_{G(O')} R \ar[r] & \mathbf 1,
     \end{tikzcd}
   \]
   Moreover,
   \[
     F^\times_{S,O'} / F^\times_{S,\calO} \simeq  \nrd(\widehat \calO'^\times)/\!\nrd(\widehat \calO^\times).
   \]
 \end{proposition}

 \begin{proof}
   The existence of the exact sequence follows again from Lemma~\ref{lem:hom-seq}, applied to
   \[
     \begin{tikzcd}
       F^\times_{S,\calO} \ar[r, "f"] & F^\times_{S,O'} \ar[r, "g"] & \Idl _S R.
     \end{tikzcd}
   \]
   Weak approximation again implies that the homomorphism $F_{S,O'}^\times \to \prod_{\fp \in S} \nrd(O'^\times_\fp) /\! \nrd(O^\times_\fp)$ is surjective, and hence gives rise to the claimed isomorphism.
 \end{proof}

 \begin{proposition} \label{prop:clgrowth}
   We have
   \[
     \frac{\card{\Cl_{G(\calO)} R}}{\card{\Cl R}} = \frac{2^n [\widehat R^\times: \nrd(\widehat \calO^\times)]}{[R^\times: R^\times \cap F^\times_{S,\calO}]} =  [R^\times_{>0} \cap \nrd(\widehat \calO^\times) : R^{\times 2}] [\widehat R^\times:\nrd(\widehat \calO^\times)].
   \]
 \end{proposition}

 \begin{proof}
   The first equality follows from Proposition~\ref{prop:cgrp-sequence}.  For the second, note that
   \[
     R^{\times 2} \subseteq R^\times \cap F^\times_{S,\calO} \subseteq R^\times,
   \]
   and, since $F$ is totally real of degree $n$, we have $[R^\times: R^{\times 2}] = 2^n$.
   Thus,
   \[
     \frac{2^n}{[R^\times : R^\times \cap F^\times_{S,\calO}]} = [R^\times \cap F^\times_{S,\calO} : R^{\times 2}].
   \]
   Moreover, $R^\times \cap F^\times_{S,\calO} = R^\times_{>0} \cap \nrd(\widehat \calO^\times)$.
   Hence, the second equality is shown.
 \end{proof}

 \begin{corollary} \label{cor:indexbound}
   Let $O \subseteq \calO'$ be orders with $\calO'_{\fp}$ maximal and $\calO$ and $\calO'$ only differing at the prime ideal $\fp$.
   Then
   \[
     \frac{\card{\Cl_{G(\calO)} R}}{\card{\Cl_{G(O')} R}} = \frac{[R_\fp^\times : \nrd(O_\fp^\times)] }{[R^\times \cap F_{S,O'}^\times:R^\times \cap F_{S,\calO}^\times]} \ \text{ divides }\ [R_\fp^\times : \nrd(O_\fp^\times)].
   \]
 \end{corollary}

 \begin{proof}
   Noting that $\nrd(O'^\times_\fp) = R^\times_\fp$, the result follows from Proposition~\ref{prop:cgrp-sequence-suborder}.
 \end{proof}

 \section{Bounds and algorithmic considerations} \label{sec:bounds}

In this section, we bound the set of definite orders with the Hermite property by an estimate of mass and Odlyzko bounds.  

\subsection*{Setup}

Let $\zeta_F(s)$ the Dedekind zeta function of $F$, let $d_F$ be the absolute discriminant of $F$, and let $\fN \colonequals \discrd(\calO)$ the reduced discriminant of $\calO$.  We abbreviate $h(R) \colonequals \card{\Cl R}$.   For a prime $\fp \mid \fN$ with $\Nm(\fp) = q$, we denote by $(O\,|\,\fp) \in \{-1,0,1\}$ the \defi{Eichler symbol} 
\cite[Definition 24.3.2; {\citealp[Definition II.2.10]{Vigneras80}}]{Voight18}, and we define \cite[26.1.1]{Voight18}
 \[
   \lambda(O,\fp)
   \colonequals \frac{1 - \Nm(\fp)^{-2}}{1 - (\calO\,|\,\frakp)\Nm(\fp)}
   =
   \begin{cases}
     1 + 1/q, &\text{if $(O\,|\,\fp) = 1$;}\\
     1 - 1/q, &\text{if $(O\,|\,\fp) = -1$;}\\
     1 - 1/q^2, &\text{if $(O\,|\,\fp) = 0$.}
   \end{cases}
 \]

We have the following generalization of Eichler's mass formula to arbitrary definite orders.

\begin{theorem}[Mass formula]  \label{thm:massformula}
We have
\begin{equation}
   \mass(\Cls\calO) = \frac{2 \zeta_F(2)}{(2\pi)^{2n}} d_F^{3/2} h(R) \Nm(\fN) \prod_{\fp \mid \fN} \lambda(O,\fp).
\end{equation}
\end{theorem}

\begin{proof}
See Vign\'eras \cite[Corollaire V.2.3]{Vigneras80} for the case of $\calO$ an Eichler order and more generally Voight \cite[Main Theorem 26.1.5, Remark 26.1.14]{Voight18} for a proof, complete references, and discussion.
\end{proof}
 
 \subsection*{Bounds}
 
 We now bound the set of Hermite orders to a finite set.  
 
\begin{proposition}[Hallouin--Maire \cite{HallouinMaire06}] \label{prop:df-bound}
  If $\calO$ is a Hermite order, then $n=[F:\QQ] \le 9$ and
  \begin{align}
    d_F^{1/n} &\le 2^{4/3 - 2/(3n)} \pi^{4/3} \bigg( \frac{\card{\Cl^+ R}}{\card{\Cl R}} \bigg)^{2/(3n)} \label{eq:df1} \\
              &    \le 2^{2 - 4/(3n)} \pi^{4/3}. \label{eq:df2}
  \end{align}
  Moreover, if $n = 9$, then $d_F^{1/9} < 13.53$.
\end{proposition}

\begin{proof}
  We repeat a simple variation of the proof of Hallouin--Maire for the convenience of the reader.

  If $\calO$ is a Hermite order, then the same is true for every order containing $\calO$, by Corollary~\ref{cor:canc-overorder}.
  Thus, we may without restriction assume that $\calO$ is maximal.
  Then $\nrd(O_\fp^\times) = R_\fp^\times$ for all primes $\fp$ of $R$, and hence $\Cl_{G(\calO)} R = \Cl^+ R$ is the narrow class group.

 In Corollary \ref{cor:equamassHermite}, we showed that $\calO$ is Hermite then 
\begin{equation}
\mass(\Cls\calO) = \frac{\card{\Cl_{G(\calO)} R}}{[O^\times:R^\times]} \leq \card{\Cl^+ R}. 
\end{equation}
Applying the mass formula (Theorem \ref{thm:massformula}) and using the trivial estimates $\zeta_F(2) \ge 1$ and $\Nm(\fN) \prod_{\fp | \fN} \lambda(O|\fp) \ge 1$, the first bound for the root discriminant \eqref{eq:df1} is obtained.

  By Proposition~\ref{prop:clgrowth}, 
  \[
    \frac{\card{\Cl^+ R}}{\card{\Cl R}} = \frac{2^n}{[R^\times : R^\times \cap F^\times_{>0}]} \le 2^{n-1}.
  \]
  Substituting this into the first bound, we get the bound \eqref{eq:df2}.

  Comparing \eqref{eq:df2} with the discriminant bounds for totally real fields found in Odlyzko's tables \cite{Odlyzko76}, we conclude $n \le 14$.

  The Hilbert class field of $F$ is a totally real field of absolute degree $nh(R)$ and with the same root discriminant as $F$.
  If $\Cl R$ is nontrivial, then the degree of the Hilbert class field is at least $2n$.
  Again comparing with the Odlyzko bounds for totally real fields, we conclude $\Cl R$ must be trivial if $n \ge 8$.
  The quotient $\Cl^+R / \Cl R$ is an elementary abelian $2$-group and the Armitage--Fr\"ohlich theorem \cite[I]{ArmitageFroehlich} states that
  \[
    \rk_2 \Cl^+ R - \rk_2 \Cl R \le \lfloor n / 2 \rfloor.
  \]
  Hence, in case $n \ge 8$, we have $\card{\Cl^+ R} / \card{ \Cl R } = \card{\Cl^+ R } \le 2^{\lfloor n/2 \rfloor}$.

  Substituting into \eqref{eq:df1}, we must have $d_F^{1/n} \le 2^{4/3 + 2(\lfloor n/2 \rfloor -1)/(3n)} \pi^{4/3}$ if $n \ge 8$.
  Comparing this improved bound with the Odlyzko tables, we find $n \le 10$.
  Moreover, $d_F^{1/10} \le 13.95$ for $n=10$ and $d_F^{1/9} < 13.53$ for $n=9$.
  By Voight's tables of totally real fields \cite{Voight08b}, there are no such fields of degree $10$.
  Hence, $n \le 9$.
\end{proof}

\begin{remark}
  Without Odlyzko's tables, Vign\'eras \cite{Vigneras76} showed $[F:\QQ] \le 33$.
  By first bounding $n \le 14$ as above and then using the bound in \eqref{eq:df1} for each such $n$ and different possibilities of $\card{\Cl^+R} / \card{\Cl R}$, Hallouin--Maire show $n \le 8$ (without using Armitage-Fr\"ohlich). Using more careful arguments, Hallouin--Maire actually show $n \le 6$ in \cite[Proposition 11]{HallouinMaire06} and for these degrees obtain improved bounds on the discriminant in \cite[Proposition 12]{HallouinMaire06}.
  Since we leave the ultimate classification up to a computer, \emph{a priori} we only need bounds that are good enough to ensure that all fields in questions have been tabulated.
  The weaker bounds that are more easily obtained suffice for us.
\end{remark}

\begin{proposition} \label{prop:disc-bound}
  If $\fD=\disc B$ and $B$ contains a Hermite order $\calO$, then
  \[
    \prod_{\fp \mid \fD} (\Nm(\fp) - 1) = \Nm(\fD) \prod_{\fp \mid \fD} \lambda(O,\fp) \le \frac{2^{2n-1} \pi^{2n}}{d_F^{3/2}} \frac{\card{\Cl^+ R}}{\card{\Cl R}}.
  \]
\end{proposition}

\begin{proof}
  As in the previous proposition, we may suppose that $\calO$ is maximal.
  Then $\lambda(O,\fp) = 1 - 1/\Nm(\fp)$ for all $\fp$ with $\fp \mid \fD$, and $\lambda(O,\fp)=1$ otherwise.
  Again the claim follows from the mass formula (Theorem \ref{thm:massformula}), together with $\mass(\Cls\calO) \leq \card{\Cl^+ R}$.
\end{proof}

\begin{lemma} \label{lem:sffindexbound}
  Let $O \subseteq \calO'$ be Hermite orders.
  Suppose further that $\calO$ and $\calO'$ only differ at the prime ideal $\fp$, and that $\calO'_\fp$ is maximal.
  Let $q = \Nm(\fp)$, let $[O'_\fp:O_\fp] = q^m$, and let $2^l = [R_\fp^\times:R_\fp^{\times 2}]$.
  Then
  \[
    [O'^\times_\fp:O_\fp^\times]
    =
    q^m \lambda(O,\fp) \cdot
    \begin{cases}
      1, & \text{if $B$ is split at $\fp$; and} \\
      (1 - 1/q)^{-1}, &\text{if $B$ is ramified at $\fp$.}
    \end{cases}
  \]
  Moreover, $[O'^\times_\fp:O_\fp^\times]$ divides $2^l [O'^\times:R^\times]$.
\end{lemma}

\begin{proof}
  For the formula for $[O'^\times_\fp:O_\fp^\times]$, see Voight \cite[Lemma 26.6.7]{Voight18}.

  Proposition \ref{prop:sff}(iv) and Theorem~\ref{thm:mass}\ref*{mass:suborder} imply
  \[
    \frac{1}{[O^\times:R^\times]} = \mass(\Cls^{[R]} O) = \frac{\card{\Cl_{G(O')} R}}{\card{\Cl_{G(\calO)} R}} [O'^\times_\fp:O_\fp^\times] \mass(\Cls^{[R]} \calO').
  \]
  Substituting $\mass(\Cls^{[R]} \calO') = [O'^\times:R^\times]^{-1}$, together with Corollary~\ref{cor:indexbound}, gives
  \[
    [O^\times:R^\times] = \frac{[R_\fp^\times:\nrd(O_\fp^\times)] [O'^\times:R^\times]}{[R^\times \cap F_{S,O'}^\times : R^\times \cap F^\times_{S,\calO}] [O'^\times_\fp:O^\times_\fp]}.
  \]
  Thus $ [O'^\times_\fp:O^\times_\fp]$ divides $[R_\fp^\times:\nrd(O_\fp^\times)] [O'^\times:R^\times]$.
  Since $R^{\times 2}_\fp \subseteq \nrd(O_\fp^\times) \subseteq R^\times_\fp$, finally $[O'^\times_\fp:O^\times_\fp]$ divides $2^l [O'^\times:R^\times]$.
\end{proof}

The previous bounds imply the following finiteness result, proven by Vign\'eras \cite{Vigneras76}.  

\begin{corollary}[{Vign\'eras \cite{Vigneras76}}] \label{cor:finitelymany}
  There exist only finitely many definite Hermite quaternion orders $\calO$.
\end{corollary}

\begin{proof}
  Since $F$ is a totally real field with bounded root discriminant by Proposition~\ref{prop:df-bound}, there exist only finitely many such fields.
  For each field, there exist only a finite number of possible choices for the ramified places of $B$ by Proposition~\ref{prop:disc-bound}.
  Finally, for each of the finitely many isomorphism classes of maximal orders in $B$, the index of a Hermite suborder is bounded by Lemma~\ref{lem:sffindexbound}.
\end{proof}

\subsection*{Algorithm}

We are now in a position to state an algorithm that finds all definite Hermite orders.

\begin{algorithm} \label{alg:theoneandonly}
  The following algorithm enumerates all definite Hermite orders.
  \begin{enumerate}
  \item[1.] Using the tabulation of totally real fields \cite{Voight08,Voight08b}, enumerate all fields with $n=[F:\QQ] \le 9$ and, $d_F^{1/n} < 16.4$ if $n \le 8$, respectively $d_F^{1/n} < 13.53$ if $n=9$.
  For each eligible field $F$, use Proposition~\ref{prop:disc-bound} to compute a finite list of primes $\fp$ at which $B$ can ramify.
  \item[2.] For each such algebra $B$, determine a set of representatives for the isomorphism classes of maximal orders.
  \item[3.] For each Hermite maximal order $\calO'$, using Lemma~\ref{lem:sffindexbound}, compute a list of prime ideals of $R$ at which we need to consider non-maximal orders.
  \item[4.] Iteratively compute suborders of $\calO'$ at the given primes, using Lemma~\ref{lem:sffindexbound} to bound the necessary index, and check them for the desired property by computing their stable class group.
      \end{enumerate}
\end{algorithm}

\begin{proof}[Proof of correctness]
The bound in Step 1 is valid by Proposition \ref{prop:df-bound}.  For Step 2, we refer to Kirschmer--Voight \cite{KirschmerVoight2,KirschmerVoight1}.  In Step 4, to check whether a given order is a Hermite ring, we compute the stable class group $\Cl_{G(\calO)} R$ and $\mass(\Cls\calO)$ and use Proposition~\ref{prop:sff}\ref{sff:mass}.
\end{proof}

The enumeration of suborders and the computation of $\Cl_{G(\calO)} R$ are not readily available in existing computer algebra systems.
Thus, we give some more detail on how these steps can be implemented efficiently.

\begin{remark}[Computation of $\Cl_{G(\calO)} R$]
  Since algorithms to compute ray class groups are already implemented in computer algebra systems, it is easiest to compute $\Cl_{G(\calO)} R$ as a quotient of such a group.

  \begin{enumerate}
  \item[1.]
    We compute the stable class group as a quotient of a ray class group.
    For a prime ideal $\fp$ of $R$, define $l(\fp) = \val_\fp(4) + 1$.
    Let $\fm=\prod_{\fp \in S} \fp^{l(\fp)}$, and let $F^\times_\fm = \{ a \in F^\times_{>0} : a \equiv 1 \mod \fm \}$.
    By choice of $l(\fp)$ and Hensel's Lemma, $1 + \fp^{l(\fp)} R_\fp \subseteq  R_\fp^{\times 2} \subseteq \nrd(O_\fp^\times)$.
    Thus, $F^\times_\fm \subseteq F^\times_{S,\calO}$ and the stable class group can be realized as a quotient of the ray class group $\Cl^+_{\fm} R = \Idl_S R / F^\times_\fm$.
    More precisely, there is an exact sequence
    \[
      \begin{tikzcd}
        F^\times_{S,\calO} / F^\times_\fm \ar[r] & \Cl^+_\fm R \ar[r] & \Cl_{G(\calO)} R \ar[r] & 1.
      \end{tikzcd}
    \]
    As in Proposition~\ref{prop:cgrp-sequence-suborder},
    \[
      F^\times_{S,\calO} / F^\times_\fm \simeq \prod_{\fp \in S} \nrd(O_\fp^\times) / (1 + \fp^{l(\fp)} R_\fp).
    \]
    To compute $\Cl_{G(\calO)} R$, we therefore first compute $\Cl^+_\fm R$, and then compute, for each $\fp \in S$, a set of generators for $\nrd(O_\fp^\times) / (1 + \fp^{l(\fp)} R_\fp)$.
    Using the Chinese Remainder Theorem to obtain suitable global representatives for these generators, we compute $\Cl_{G(\calO)} R$ as quotient of $\Cl^+_\fm R$.
 
  \item[2.] To compute $\nrd(O_\fp^\times) / (1 + \fp^{l(\fp)} R_\fp)$, first note that $\nrd(1 + \fp^{l(\fp)} O_\fp) \subseteq 1 + \fp^{l(\fp)} R_\fp$.
    Hence, the reduced norm induces a homomorphism
    \[
      \begin{tikzcd}
        O_\fp^\times/(1 + \fp^{l(\fp)} O_\fp)  \ar[r] & R_\fp^\times\!/ (1 + \fp^{l(\fp)} R_\fp).
      \end{tikzcd}
    \]
    Thus, it suffices to compute the image of a generating set of $O_\fp^\times / (1+\fp^{l(\fp)} O_\fp)$ under $\nrd$.

    Since
    \[
      O_\fp^\times  \supseteq 1+ \fp O_\fp \supsetneq 1 + \fp^2 O_\fp \supsetneq \dots \supsetneq 1 + \fp^{l(\fp)} O_\fp,
    \]
    it suffices to compute generating sets of the multiplicative groups $O_\fp^\times / (1 + \fp O_\fp) \simeq (O_\fp / \fp O_\fp)^\times \simeq (O/ \fp O)^\times$ and $(1 + \fp^i O_\fp) / (1 + \fp^{i+1} O_\fp)$.
    The latter group is isomorphic to the additive group $O/ \fp O$.
    A generating set for $(O/\fp O)^\times$ can be computed since $O/\fp O$ is a finite-dimensional algebra over the finite field $R / \fp$.
    A $\ZZ$-basis of $\calO$ yields a generating set for the additive abelian group $O/\fp O$.
  \end{enumerate}
\end{remark}

\begin{remark}[Enumeration of suborders]
  To enumerate suborders in a systematic way, we organize them by radical idealizers \cite[Section 24.4]{Voight18}.
  We proceed as follows.
  \begin{enumerate}
  \item[1.] Compute representatives for the isomorphism classes of all maximal orders, giving a list of orders.
  \item[2.] For each prime $\fp$ for which we need to consider non-maximal orders, and for each ($\fp$-maximal) order $\calO$ computed so far:
    \begin{enumroman}
    \item[a.] First compute the hereditary suborders that are non-maximal at $\fp$.
    \item[b.] Recursively compute all suborders whose radical idealizer at $\fp$ is one of the orders computed so far.
    \end{enumroman}
  \end{enumerate}

  For a given prime $\fp$, this procedure produces a tree of orders.
  If an order $\calO$ exceeds the index bound or  fails to be Hermite, we need not check its suborders anymore due to Corollary~\ref{cor:canc-overorder}.
  
  If $\calO'$ is an order, and $O \subseteq \calO'$ is a suborder whose radical idealizer at $\fp$ is $\calO'$, then $\fp \calO' \subseteq O$.
  Hence, to compute candidate orders $\calO$, we simply check the preimages of all the subrings of the (finite) $R/\fp$-algebra $\calO'/\fp \calO'$.
  We use isomorphism testing between those orders to enumerate the orders up to isomorphism.
\end{remark}

We have implemented this algorithm in Magma \cite{Magma}.
Our code is available on the web \cite{Smertnig-Voight19:github}.
Running this classification gives all definite quaternion orders over a ring of algebraic integers that are Hermite rings.
From this list it is easy to filter the ones having locally free cancellation.

\begin{remark} \label{rmk:wechecked}
We checked our results against existing lists.
\begin{itemize}
 \item The hereditary definite quaternion orders that are Hermite rings have been classified by Vign\'eras \cite{Vigneras76}, Hallouin--Maire \cite{HallouinMaire06}, and Smertnig \cite{Smertnig15}.
  Up to $R$-algebra isomorphism there are 168 such orders, of which 149 have locally free cancellation.
  Our list is consistent with the (corrected) old classification.

\item  Kirschmer--Lorch \cite{KirschmerLorch16} computed all definite quaternion orders with type number at most $2$, and made them available in electronic form \cite{KirschmerLorch16b}.
  This includes all orders $\calO$ with $\card{\Cls\calO} = 1$, and such an order trivially has locally free cancellation.
  Our list is consistent with theirs.

\item  Estes--Nipp \cite[Table I]{EstesNipp89} list the 40 definite Hermite quaternion $\Z$-orders.
  We found the same number of $\Z$-orders, and we matched their discriminants to the ones appearing in our list.
  \end{itemize}
\end{remark}
\begin{corollary}
  Up to $R$-algebra isomorphism, there are $375$ definite Hermite quaternion orders; of these, $316$ have locally free cancellation.
\end{corollary}
Reducing this list to orders up to ring isomorphism by Galois automorphisms, we obtain Theorem~\ref{thm:mainthm}.
Invariants describing these orders are given in Appendix \ref{appendix:tables}.
A computer-readable file containing all the orders is available on the web \cite{Smertnig-Voight19:github}.

Looking over the list in Appendix \ref{appendix:tables}, we find the following corollary.

 \begin{corollary}
   If a definite quaternion order $\calO$ has locally free cancellation, then the base ring $R$ is factorial.
 \end{corollary}

 Note that the analogous statement fails to hold for Hermite orders, with the sole exceptions being two orders over $R=\ZZ[\sqrt{15}]$.

\section{Examples and applications} \label{sec:examappli}

\subsection*{Examples}

To give a flavor of the data, we pick out two interesting examples of definite Hermite quaternion orders.
We recall the fundamental fact that there is a bijection between (twisted) similarity classes of ternary quadratic forms and quaternion orders \cite[Chapter 22]{Voight18}: this bijection has a particularly rich history, see Voight \cite[Remark 22.6.20]{Voight18} for a complete bibliography.  
%  \cite{Brzezinski82,GrossLucianovic09,Lemurell11,Voight11}.
% The correspondence as we make use of it here is presented in \cite[Chapter 22]{Voight18}.

\begin{example}
First, let $B$ be the usual Hamiltonian quaternion algebra with $i^2=j^2=-1$ and $ij=-ji$.  Then $\calO = \Z + 2\Z i + 2\Z j + 2\Z ij$ is an order of reduced discriminant $32$ that has locally free cancellation and is non-Gorenstein. 
It has $\card{\Cls O}=\card{\StCl O} = 2$ but type number $1$.  Under the bijection between isomorphism classes of quaternion orders and similarity classes of  ternary quadratic forms, $\calO$ corresponds to $2(x^2 + y^2 + z^2)$.
Its Gorenstein saturation, corresponding to $x^2 + y^2 + z^2$, is the Lipschitz order $\Z+\Z i + \Z j + \Z ij$.
\end{example}

There is also an interesting connection with quaternary quadratic forms, discussed in more detail in the next subsection.  We restrict to the case $R=\Z$ for illustration.  Let $\calO$ be a definite quaternion $\Z$-order. Then reduced norm $\nrd\!|_\calO \colon \calO \to \Z$ is a positive definite integral quaternary quadratic form.  Let $\Gen(\nrd\!|_\calO)$ be its genus and $\Cls(\nrd\!|_\calO)$ the set of isometry classes in the genus; and let $\SpGen(\nrd\!|_\calO)$ its spinor genus and $\SpCls(\nrd\!|_\calO)$ the set of isometry classes in the spinor genus.  We have $\SpCls(\nrd\!|_\calO) \subseteq \Cls(\nrd\!|_\calO)$.  Then $\calO$ is Hermite if and only $\#\SpCls(\nrd\!|_\calO)=1$, that is, its spinor genus consists of a single isometry class \cite[Theorem~1]{EstesNipp89}, and we say $\nrd\!|_\calO$ has \defi{spinor class number} $1$.  By Proposition \ref{prop:canc-sff}, it follows that $\calO$ has locally free cancellation if for every order $\calO'$ locally isomorphic to $\calO$ we have $\#\SpCls(\nrd\!|_\calO)=1$.  Parks \cite{Parks74} classified all 40 definite quaternion $\Z$-orders with spinor class number $1$, extending the list of 39 orders with $\#\Cls(\nrd\!|_\calO)=1$ (that is, class number $1$) previously determined by Pall \cite{Pall46}, with the sole outlier described as follows.

\begin{example}
Among the 40 definite Hermite $\Z$-orders, there is exactly one order $\calO$ not having locally free cancellation.
It is the unique definite order with $\#\SpCls(\nrd\!|_\calO)=1$ but $\#\Cls(\nrd\!|_\calO) \neq 1$ found by Parks.
The order $\calO$ is a Bass order of reduced discriminant $27$ in the definite quaternion algebra of discriminant $3$.
With $i^2=-3$, $j^2=-1$, and $ij=-ji$, it can be represented by $\calO = \Z + \Z(\frac{1}{2} + \frac{3}{2}i) +  \Z(3j) + \Z(\frac{3}{2}j + \frac{1}{2}k)$.
Now $\card{\Cls O} = 4$ while $\card{\StCl O}=2$.
Under the surjective map $\st\colon \Cls O \to \StCl O$, one class maps to the trivial class, while the other three classes map to the nontrivial one.
This implies that $\calO$ is Hermite; see Proposition~\ref{prop:sff}.
The order $\calO$ has type number $2$, and therefore there exists a nonisomorphic but locally isomorphic order $\calO'$ that is not Hermite. Accordingly, in $\st\colon \Cls \calO' \to \StCl \calO'$, three classes map to the trivial class, while only one maps to the nontrivial one.

The order $\calO$ corresponds to the ternary quadratic form
\[
  x^2 + 3y^2 + 3z^2 - 3yz;
\]
On the other hand, the quaternary quadratic form $\nrd\!|_\calO$ has discriminant $729=27^2$ and in the basis above is given by
\begin{equation} \label{eqn:Qparks}
Q(x,y,z,w)=x^2 + 7y^2 + 9z^2 + 3w^2 + xy + 9zw.
\end{equation}
We have $\SpCls(Q)=1$ but $\Cls(Q) \neq 1$, indeed $\Cls(Q)=3$ and $\Gen(Q)$ splits into two spinor genera.  In fact, the form $Q$ represents the \emph{unique} class of primitive quaternary quadratic form with $\#\SpCls(Q)=1$ but $\#\Cls(Q) \neq 1$ \cite{EarnestHaensch18}.
\end{example}

More generally, the classification of definite quadratic integral lattices with class number $1$, started by Watson in the 1960s, has recently been finished by Lorch--Kirschmer \cite{LorchKirschmer13}. (The rank $2$ case assumes GRH.)
Earnest--Haensch \cite{EarnestHaensch18} conclude that the lattice found by Parks remains the sole example with spinor class number $1$ but not class number $1$ (by completing the classification for rank $4$).

\subsection*{An application to factorizations}

Aside from the intrinsic importance of the Hermite and cancellation properties for the description of isomorphism classes of locally free modules, the Hermite property has been shown to have important consequences for the factorizations of elements in an order.
This has been observed by Estes--Nipp in \cite{EstesNipp89,Estes91a} in their study of \emph{factorizations induced by norm factorization \textup{(}FNF\textup{)}}, as well as more recently in studying non-unique factorizations in orders by means of arithmetical invariants.

We give a brief glimpse at this connection, by highlighting some properties of the sets of lengths.
Let $\calO^\bullet$ be the multiplicative monoid of non-zero-divisors of $\calO$, that is, the elements of non-zero reduced norm.
An element $u \in \calO^\bullet$ is an \defi{atom} if it cannot be expressed as a product of two non-units.
Every non-unit $a \in \calO^\bullet$ can be written as a product of atoms, but in general not uniquely so.
If $a=u_1\cdots u_k$, with atoms $u_i$, then $k$ is a \defi{length} of $a$, and we write $\sL(a) \subseteq \Z_{\ge 0}$ for the \defi{set of lengths} of $a$.
By considering the reduced norm, it is easily seen that these sets are finite.
If $\card{\sL(a)} \ge 2$, then $\card{\sL(a^k)} \ge k+1$.
Hence, the sets of lengths are either all singletons, in which case $\calO^\bullet$ is \defi{half-factorial}, or become arbitrarily large.
If $\sL(a) = \{ l_1 < \ldots < l_m\}$, let $\Delta(a) = \{ l_2 - l_1, \ldots, l_m - l_{m-1} \}$, and let $\Delta(\calO^\bullet) = \bigcup_{a \in \calO^\bullet} \Delta(a)$ denote the \defi{set of distances} of the monoid $\calO^\bullet$.

If $\calO$ is a Hermite hereditary order, then $\calO^\bullet$ is a \defi{transfer Krull monoid of finite type}.
This follows from \cite[Theorem 4.4]{Smertnig16}, after recognizing the occurring class group as isomorphic to $\StCl \calO$.
For maximal orders, this is more explicit in \cite[Theorem 1.1]{Smertnig13}.
For hereditary orders satisfying the Eichler condition, that is $\dim_F B > 4$ or $B$ is an indefinite quaternion algebra, the result can also be obtained from a theorem of Estes \cite{Estes91b}.

The fact that $\calO$ is a transfer Krull monoid of finite type, implies that many questions on factorizations in $\calO$, in particular all the ones on sets of lengths, can be reduced to questions in combinatorial and additive number theory over finite abelian groups, specifically the stable class group $\StCl \calO$.
See the surveys \cite{Geroldinger16, Schmid16} as a starting point into the extensive literature; and \cite{Tringali17, GeroldingerZhong18, Zhong18, GeroldingerSchmid19} for recent progress.
In particular, the set of distances $\Delta(O^\bullet)$ is finite, indeed $\Delta(O^\bullet) = \{1, \ldots, D\}$ for some $D \in \Z_{\ge 0}$.
The sets $\sL(a)$ satisfy the \defi{Structure Theorem for Sets of Lengths} \cite[Theorem 2.6 and Definition 2.5]{Geroldinger16}, roughly saying that each $\sL(a)$ is a finite union of (finite) arithmetical progressions with distances $d \in \Delta(\calO^\bullet)$, and possibly some gaps at the beginning and the end.

Let us now restrict to $\calO$ a maximal order.
Then, if $\calO$ is \emph{not} Hermite, it is known that $\calO$ is \emph{not} a transfer Krull monoid \cite[Theorem 1.2]{Smertnig13}.
Not just does this structure break down, the factorization properties of $\calO$ are totally different.
For instance, although $\Delta(\calO^\bullet)$ was finite before, now $\Delta(\calO^\bullet) = \Z_{\ge 1}$.
Indeed, for any $l \ge 0$, there exist elements $a \in \calO^\bullet$ with $\sL(a) = \{3\} \cup l+E$, where $E$ is a non-empty subset of $\{2,3,4\}$ \cite[Proposition 7.2]{Smertnig13}.
Thus, the Hermite property provides a sharp dividing line between two completely different regimes as far as the factorizations of elements are concerned.

 \appendix
 
 \section{Comparison with the criterion of Vign\'eras} \label{appendix:comparvign}

 To check whether an order is a Hermite ring, we have checked if the equality 
 \[ \mass(\Cls^{[R]} O)=[O^\times:R^\times]^{-1} \] 
 holds.
 This is very close to Vign\'eras's criterion \cite{Vigneras76} but not precisely the same.
 In this appendix, we show how to derive her original criterion from ours.

 \begin{theorem}[{Vign\'eras \cite[Th\'eor\`eme 3]{Vigneras76}}] \label{thm:vigneras-criterion}
  A definite order $\calO \subseteq B$ is a Hermite ring if and only if
  \begin{equation} \label{eq:vigneras}
    2[O^\times : R^\times] = \tau(\widehat \calO^1) [F^\times_{>0} \cap \nrd(\widehat \calO^\times):R^{\times 2}],
  \end{equation}
  where $\tau$ is the Tamagawa measure.
  For a hereditary order $\calO$, writing $\mathfrak N = \mathfrak D \mathfrak M$ with $\mathfrak D = \disc B$, we have
  \[
    \tau(\widehat \calO^1)^{-1} = 2^{-n} \zeta_F(-1) \prod_{\fp \mid \mathfrak D} (1 - \Nm \fp) \prod_{\fp \mid \mathfrak M} (1 + \Nm \fp)
  \]
\end{theorem}

We will need the following small lemma.
\begin{lemma} \label{lem:norm1-index}
  Let $O \subseteq \calO'$ be orders and let $\fp$ be a prime ideal of $R$.
  Then
  \[
    \frac{[O_\fp'^\times:O_\fp^\times]}{[\nrd(O'^\times_\fp):\nrd(O_\fp^\times)]} = [O'^1_\fp : O_\fp^1].
  \]
\end{lemma}

\begin{proof}
  Consider the diagram
  \[
    \begin{tikzcd}
      1 \ar[r] & O_\fp^1 \ar[r] \ar[d] & O_\fp^\times \ar[r, "\nrd"] \ar[d] & \nrd(O_\fp^\times) \ar[r] \ar[d] & 1 \\
      1 \ar[r] & \calO'^1_\fp \ar[r] & \calO'^\times_\fp \ar[r, "\nrd"] & \nrd(O'^\times_\fp) \ar[r] & 1.
    \end{tikzcd}
  \]
  By the snake lemma, there is a short exact sequence of the cokernels
  \[
    \begin{tikzcd}
      1 \ar[r] & \calO'^1_\fp / O_\fp^1 \ar[r] & \calO'^\times_\fp /O^\times_\fp \ar[r] & \nrd(O'^\times_\fp) / \nrd(O_\fp^\times) \ar[r] & 1
    \end{tikzcd}
  \]
  which gives the result.
\end{proof}

\begin{proof}[Proof of Theorem~\textup{\ref{thm:vigneras-criterion}}]
  By Proposition \ref{prop:sff}, $\calO$ is a Hermite ring if and only if $\mass(\Cls^{[R]} O) = [O^\times:R^\times]^{-1}$.  Therefore, we need to show
  \begin{equation} \label{eq:mass-tamagawa}
    \mass(\Cls^{[R]} O) = 2\tau(\widehat \calO^1)^{-1} [F^\times_{>0} \cap \nrd(\widehat \calO^\times) : R^{\times 2}]^{-1}.
  \end{equation}

  We first consider the case where $\calO$ is hereditary.
  From the mass formula (Theorem \ref{thm:massformula}), and Proposition~\ref{prop:clgrowth},
  \[
    \mass(\Cls^{[R]} O)
    = \frac{\mass(\Cls\calO)}{\card{\Cl_{G(\calO)} R}}
    =\frac
    {\abs{\zeta_F(-1)} \Nm(\mathfrak N) \prod_{\fp \mid \mathfrak N} \lambda(O,\fp)}
    {2^{n-1} [R^\times_{>0} \cap \nrd(\widehat \calO^\times) : R^{\times 2}] [\widehat R^\times:\nrd(\widehat \calO^\times)]}.
  \]
  Since $\calO$ is hereditary,
  \[
    \Nm(\mathfrak N) \prod_{\fp \mid \mathfrak N} \lambda(O,\mathfrak p) = \prod_{\fp \mid \mathfrak D} (\Nm \fp - 1) \prod_{\fp \mid \mathfrak M} (\Nm \fp + 1) = (-1)^n \prod_{\fp \mid \mathfrak D} (1 - \Nm \fp) \prod_{\fp \mid \mathfrak M} (1 + \Nm \fp).
  \]
  Here, the sign can be expressed as $(-1)^n$, because $B$ is ramified at an even number of places, since $B$ is  definite it is ramified at all archimedean places, and $\fD$ is the product of all non-archimedean places at which $B$ is ramified.
  Moreover, we have $\zeta_F(-1) = \abs{\zeta_F(-1)} (-1)^n$, we have $F^\times_{>0} \cap \nrd(\widehat \calO^\times) = R^\times_{>0} \cap  \nrd(\widehat \calO^\times)$, and finally $\nrd(O^\times_\fp) = R^\times_\fp$ for all prime ideals $\fp$ (since $\calO$ is an Eichler order).
  Comparing with the expression for $\tau(\widehat \calO^1)^{-1}$, we observe that \eqref{eq:mass-tamagawa} holds.

  Now suppose that $\calO$ is not hereditary, and let $\calO'$ be a hereditary order containing $\calO$.
  Then $\mass(\Cls^{[R]} \calO') = 2 \tau(\widehat \calO'^1)^{-1}[R^\times_{>0} \colon R^{\times 2}]^{-1}$ by what we already showed.
  Now
  \[
    \tau(\widehat \calO^1)^{-1} = [\widehat \calO'^1: \widehat \calO^1] \tau(\widehat \calO'^1)^{-1}.
  \]
  On the other hand, 
  \[
    \mass(\Cls^{[R]} O) = \frac{\card{\Cl^+ R}}{\card{\Cl_{G(\calO)} R}} [\widehat \calO'^\times: \widehat \calO^\times] \mass(\Cls^{[R]} \calO'),
  \]
  by Theorem~\ref{thm:mass}.
  Now,
  \[
    \frac{\card{\Cl_{G(\calO)} R}}{\card{\Cl^+ R}} = \frac{[R^\times_{>0} \cap \nrd(\widehat \calO^\times) : R^{\times 2}]}{[R^\times_{>0}:R^{\times 2}]} [\widehat R^\times :\nrd(\widehat \calO^\times)]
  \]
  by Proposition~\ref{prop:clgrowth}.
  By Lemma~\ref{lem:norm1-index},
  \[
    [\widehat \calO'^1: \widehat \calO^1]  = \frac{[\widehat \calO'^\times:\widehat \calO^\times] }{[\widehat R^\times :\nrd(\widehat \calO^\times)]}.
  \]
  Putting everything together, the claim follows.
\end{proof}

% Appendix that includes invariants for all definite Hermite orders
\section{Tables} \label{appendix:tables}

The following tables list invariants describing all definite quaternion orders $\calO$ that are Hermite rings, hence all orders with locally free cancellation.
For simplicity, orders are listed up to ring isomorphism not up to $R$-algebra isomorphism.
That is, we use automorphisms of the base field to identify some of the orders.
The corresponding multiplicity is listed in the last column of the tables.
Thus, there are 303 entries in the tables, but 375 orders up to $R$-algebra automorphism.

For example, in $R=\Z[(1+\sqrt{5})/2]$ the prime number $11$ splits, so that $11 R = \fp \fq$.
Let $B$ be the quaternion algebra over $F=\Q[\sqrt{5}]$ that is ramified only at the archimedean primes (corresponding to $n=2$, $d=5$, $D=1$ in the table below).
In $B$ there exist two hereditary orders $O$, $O'$ with $N\coloneqq\Nm(\discrd(O))=\Nm(\discrd(O'))=11$ having cancellation.
One of these is maximal at $\fp$ and non-maximal at $\fq$, while the other one is non-maximal at $\fp$ and maximal at $\fq$.
However, the Galois automorphism of $F$ maps $\fp$ to $\fq$, and, extending it to a ring automorphism of $B$, it maps $O$ to $O'$.
Thus, $O$ and $O'$ are isomorphic as rings but not as $R$-algebras.
We record only one entry in the table (the line with $n=2$, $d=5$, $D=1$, $N=11$), but note the multiplicity $2$ in the column labeled by `\#'.

In the tables, we use the following notation.
\begin{itemize}
\item $n=[F:\Q]$ is the degree of the base field $F$.
\item $d$ is the discriminant of $F$.
\item $D=\Nm_{F|\Q}(\disc B)$ is the norm of the discriminant of the quaternion algebra $B$ over $F$.
\item $N=\Nm_{F|\Q}(\discrd \calO)$ is the norm of the reduced discriminant of the order $\calO \subseteq B$.
\end{itemize}
An empty entry in one of these columns means that the corresponding structure (quaternion algebra, respectively, the base field), is the same as in the previous line.
For the number fields appearing in the tables, the pair $(n,d)$ uniquely characterizes them up to field isomorphism.

For a prime $\fp$ of $R$, let $k \colonequals R/\frakp$ be the residue class field of $\fp$.
Recall that the \defi{Eichler symbol} of a quaternion $R$-order $\calO$ is defined as
\[
  (O\,|\,\fp) = \left( \frac{O}{\fp} \right) =
  \begin{cases}
    *, & \text{if $O_\fp/\rad O_\fp \simeq \M_2(k)$;} \\
    1, & \text{if $O_\fp/\rad O_\fp \simeq k \times k$;}\\
    0, & \text{if $O_\fp/\rad O_\fp \simeq k$; and}\\
    -1, &\text{if $O_\fp/\rad O_\fp$ is a (separable) quadratic field extension of $k$.}
  \end{cases}
\]
By \defi{residually inert}, respectively, \defi{residually quadratic}, we mean an order for which $(O\,|\,\fp) \in \{*,-1\}$, respectively, $(O\,|\,\fp) \in \{*, \pm 1\}$ for all prime ideals $\fp$ of $R$.

The next column lists the strongest property that the order possesses (globally), among the following:
{\scriptsize
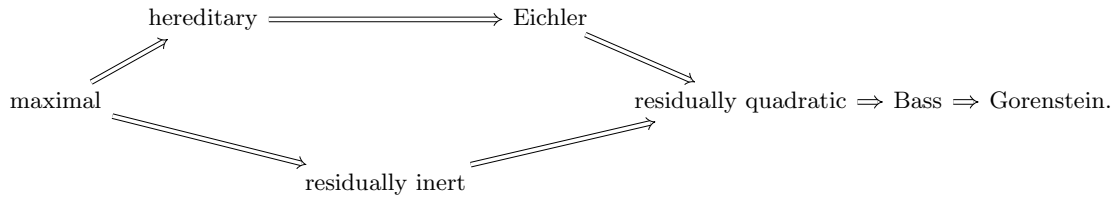
  \[
    \begin{tikzcd}[every arrow/.append style=Rightarrow, column sep=10pt]
      & \text{hereditary} \ar[rr] & & \text{Eichler} \ar[rd] &&  &\\
      \text{maximal} \ar[ru] \ar[rrd] & & & & \text{residually quadratic} \ar[r] & \text{Bass} \ar[r] & \text{Gorenstein}. \\
      & & \text{residually inert} \ar[rru] & & & &
    \end{tikzcd}
  \]
}

The column labeled `c' contains an entry `c' whenever the order listed has locally free cancellation, and is empty otherwise (that is, the order is a Hermite order but does not have locally free cancellation).

Next we list local Eichler symbols.
The data are organized by rational primes.
For each prime $p$, if there is a prime ideal $\fp$ of $R$ over $p$ for which $O_\fp$ is not maximal, we list the local Eichler symbol $(O\,|\,\fp)$ for all $\fp$ over $p$.
Here, the primes lying over $p$ are sorted first in ascending order of $\Nm \fp$, then by ascending ramification index.
For easier readability, if $\calO$ is maximal at every $\fp$ containing $p$, we do not list any data.
This means $(O\,|\,\fp) = *$ if the quaternion algebra is unramified at $\fp$, and $(O\,|\,\fp)=-1$ if the algebra is ramified at $\fp$.

Next, we list the cardinality of the class set of locally free right ideals of $\calO$, of the stable class group of $\calO$ (which is isomorphic to $\Cl_{G(\calO)} R$), and the class group of $R$ itself.
If the value is $1$, we omit the entry for easier readability.
Next, we list the cardinality of the class set of locally free right ideals of $\calO$, of the stable class group of $\calO$ (which is isomorphic to $\Cl_{G(\calO)} R$), the type number $t(O)$, and the class group of $R$ itself.
If the value is $1$, we omit the entry for easier readability.

The type number is the number of isomorphism classes of orders that are locally isomorphic to $O$; again we suppress the value $1$.
Observe that if $t(O) > 1$ and $O$ has cancellation, then the orders that are locally isomorphic but not isomorphic to $O$ also have cancellation, and therefore also appear in the tables.
Several instances of this can be seen over in the algebra with $N=1$ over the field $\Q[\sqrt{12}]$.

The final column lists the multiplicity with which the entry should be counted, to count the orders up to $R$-algebra isomorphism. Again, if the value is $1$, we omit the entry.

The tables are intended to give an overview over the orders, not to characterize them up to isomorphism.
A computer-readable list of all the orders in the table, including generators, is available electronically at \cite{Smertnig-Voight19:github}.

\rowcolors{2}{}{lightgray!25}
\begin{table*}
  \caption{Definite Hermite quaternion orders over $\ZZ$.}
\begin{tabular}{rr|c|c|cccc|ccc}
  $D$ & $N$ &                  & c & 2  & 3  & 5  & 11 & $\#\Cls(O)$ & $\#\StCl(O)$ & $t(O)$ \\
  \toprule
  2   & 2   & maximal          & c & -1 &    &    &    &             &              &        \\
      & 4   & Bass             & c & 0  &    &    &    &             &              &        \\
      & 6   & hereditary       & c & -1 & 1  &    &    &             &              &        \\
      & 8   & Bass             & c & 0  &    &    &    &             &              &        \\
      & 8   & residually inert & c & -1 &    &    &    &             &              &        \\
      & 10  & hereditary       & c & -1 &    & 1  &    &             &              &        \\
      & 12  & Bass             & c & 0  & 1  &    &    &             &              &        \\
\end{tabular}

\end{table*}

\begin{landscape}
\setlength{\tabcolsep}{3pt}
\rowcolors{2}{}{lightgray!25}
% this will make sure that the colors are really alternating and there is no white space
% around the midrules
\renewcommand{\midrule}{%
  \specialrule{\lightrulewidth}{0pt}{0pt}%
  \specialrule{0pt}{0pt}{0pt}}

  {
   \scriptsize
  \begin{longtable}{rrrr|c|c|ccccccccc|cccc|c}
  \caption{Definite Hermite quaternion orders over rings of algebraic integers in number fields of degree $n\ge 2$.}\\
$n$ & $d$     & $D$ & $N$ &                      & c & 2      & 3     & 5        & 7        & 11      & 19       & 23   & 29         & 59         & $\#\Cls(O)$ & $\#\StCl(O)$ & $t(O)$ & $\#\Cl(R)$ & \# \\
\toprule
\endhead
2   & 5       & 1   & 1   & maximal              & c &        &       &          &          &         &          &      &            &            &             &              &        &            &    \\
    &         &     & 4   & hereditary           & c & 1      &       &          &          &         &          &      &            &            &             &              &        &            &    \\
    &         &     & 5   & hereditary           & c &        &       & 1        &          &         &          &      &            &            &             &              &        &            &    \\
    &         &     & 9   & hereditary           & c &        & 1     &          &          &         &          &      &            &            &             &              &        &            &    \\
    &         &     & 11  & hereditary           & c &        &       &          &          & 1, *    &          &      &            &            &             &              &        &            & 2  \\
    &         &     & 16  & Bass                 & c & 0      &       &          &          &         &          &      &            &            &             &              &        &            &    \\
    &         &     & 16  & Eichler              & c & 1      &       &          &          &         &          &      &            &            &             &              &        &            &    \\
    &         &     & 16  & residually inert     & c & -1     &       &          &          &         &          &      &            &            &             &              &        &            &    \\
\end{longtable}
  }
\end{landscape}

% Appendix that includes invariants only for definite Hermite ZZ-orders
% \include{appendix_short}

\bibliographystyle{hyperalphaabbr}
\bibliography{defcancel}

\end{document}